\documentclass{amsart}
\usepackage{graphicx}
\usepackage[margin=1.3in]{geometry}
\usepackage{amssymb}
\usepackage{amsthm}
\usepackage{tikz-cd}
\usepackage{amsmath}
\usepackage{mathtools}
\usepackage[dvipsnames]{xcolor}
\usepackage{todonotes}
\usepackage{bbold}
\usepackage{mathrsfs}
\usepackage{eucal}
\usepackage[%
    colorlinks,%
    bookmarksdepth=3,%
    citecolor=blue%
]{hyperref}

\usepackage[
backend=biber,
style=alphabetic,
sorting=nyt,
maxbibnames=99,
]{biblatex}
\usepackage{thmtools}
\usepackage{cleveref}

\newtheorem{theorem}{Theorem}[section]
\newtheorem{prop}[theorem]{Proposition}
\newtheorem{lemma}[theorem]{Lemma}
\newtheorem{cor}[theorem]{Corollary}

\theoremstyle{definition}
\newtheorem{definition}[theorem]{Definition}
\theoremstyle{remark}
\newtheorem{remark}[theorem]{Remark}
\newtheorem{example}[theorem]{Example}
\newtheorem{hypothesis}[theorem]{Hypothesis}

\newcommand{\Z}{
\mathbb{Z}
}

\newcommand{\R}{
\mathbb{R}
}

\newcommand{\N}{
\mathbb{N}
}
\newcommand{\A}{
\mathbb{A}
}
\newcommand{\PP}{
\mathbb{P}
}
\newcommand{\F}{
\mathbb{F}
}
\newcommand{\K}{
\mathsf{K}
}
\newcommand{\D}{
\mathsf{D}
}
\newcommand{\unit}{
\mathbb{1}
}
\newcommand{\cone}{
\operatorname{cone}
}
\newcommand{\id}{
\operatorname{id}
}
\newcommand{\Hom}{
\operatorname{Hom}
}
\newcommand{\Map}{
\operatorname{Map}
}
\newcommand{\End}{
\operatorname{End}
}

\newcommand{\colim}{
\operatorname{colim}
}
\newcommand{\hocolim}{
\operatorname{hocolim}
}

\newcommand{\Spec}{
\operatorname{Spec}
}
\newcommand{\Spc}{
\operatorname{Spc}
}

\newcommand{\supp}{
\operatorname{supp}
}
\newcommand{\Supp}{
\operatorname{Supp}
}
\newcommand{\ideal}[1]{
\langle #1 \rangle
}
\newcommand{\Loc}{
\operatorname{Loc}
}

\newcommand{\Sch}{
\mathsf{Sch}
}
\newcommand{\DM}{
\mathsf{DM}
}
\newcommand{\DTM}{
\mathsf{DTM}
}
\newcommand{\DAM}{
\mathsf{DAM}
}
\newcommand{\DATM}{
\mathsf{DATM}
}
\newcommand{\Chow}{
\mathsf{Chow}
}

\newcommand{\gm}{
\mathsf{gm}
}

\title{On the tensor-triangular geometry of isotropic motives}
\author{Fraser Sparks}
\address{School of Mathematical Sciences, University of Nottingham, Nottingham, NG7 2RD, UK}
\email{fraser.sparks@nottingham.ac.uk}
\date{\today}

\begin{document}

\begin{abstract}
    We investigate the tensor-triangular geometry of the categories of isotropic Tate motives, isotropic Artin motives and isotropic Artin--Tate motives. In particular, we study the categories $\DTM_\gm(k/k;\F_2)$, $\DAM_\gm(k/k;\F_2)$ and $\DATM_\gm(k/k;\F_2)$ where $k$ is a flexible field and we fix $\F_2$-coefficients. In this case, we prove their Balmer spectra are singletons. The proof of this relies on the fact that the categories in question are generated by objects whose morphisms are `nilpotent enough.' Furthermore, we study their homological spectra, and verify Balmer's \emph{Nerves of Steel} conjecture in these cases by showing these spaces are also singletons, and we give an explicit description of the homological primes. We also explore stratification, and we conclude by examining generation properties of these categories---the latter results we see also apply to global (i.e. non-isotropic) motives.
\end{abstract}

\maketitle

\Small
\textbf{MSC Subject Classification:} \texttt{14C15, 14F42, 18F99}
\normalsize
\section{Introduction}
Motives are objects that encapsulate the cohomological data of algebraic varieties. They were introduced by Grothendieck in the 1960's, where he introduced the notion of Chow motives---these are motives of smooth projective varieties and direct summands thereof. He envisioned a general theory of (mixed) motives, and much work has been done in trying to understand the motivic mysteries that still prevail to this day. In \cite{Voevodsky:TriangulatedMotives}, Voevodsky introduced his triangulated category of motives over a field $k$, denoted $\DM(k)$. It is a tensor-triangulated category, that is, a triangulated category with a biexact symmetric monoidal structure. It allows one to work with more general motives, as well as allowing the full force of triangulated categories and homological algebra to come into play. This category is well studied, yet still furnished with many open problems and conjectures, and it is much less well-understood than its topological counterpart---the derived category of abelian groups. In particular, $\DM(k)$ is much richer and more complicated than $\D(\Z)$. One measure of the complexity of tensor-triangulated (or `tt-') categories is the \emph{Balmer spectrum}, which we describe below.

As mentioned above, a tt-category is a triangulated category with a compatible symmetric monoidal structure: familiar examples will include the derived category of $R$-modules for a commutative unital ring $R$, and more generally the derived category of a scheme $X$ (its derived category of $\mathcal{O}_X$-modules with quasi-coherent cohomology sheaves). There are numerous other examples, most notably arising within stable homotopy theory and modular representation theory. Safe to say, these tt-categories appear widely throughout pure mathematics. In \cite{Balmer:tt-geometry}, Balmer initiated the study of tt-geometry: this is where one studies tt-categories $\mathcal{K}$ geometrically, through the use of the so-called \emph{Balmer spectrum}, denoted $\Spc(\mathcal{K})$. The Balmer spectrum is a topological space one assigns to a tt-category in an analogous fashion to how one assigns the Zariski spectrum to a commutative ring, and it gives the classification of particular subcategories of your tt-category $\mathcal{K}$. More precisely, the spectrum comes equipped with a notion of support, and this induces an explicit bijection between the Thomason subsets of $\Spc(\mathcal{K})$ (that is, those subsets $X \subseteq \Spc(\mathcal{K})$ which can be written as the union of closed sets, each having quasi-compact open complement) and the lattice of radical tt-ideals of $\mathcal{K}$. This allows one to determine when two objects $a,b \in \mathcal{K}$ generate the same radical tt-ideal: this happens iff $\supp(a)=\supp(b) \subseteq \Spc(\mathcal{K})$.

With regards to the study of motives, it is an open question in tt-geometry to completely describe the spectrum of (the compact part of) $\DM(k)$. Complete classifications of certain subcategories have been achieved, for instance in \cite{Peter:ArtinTate}, \cite{Gallauer:TateSpectrum} and \cite{Balmer:Artin}: these deal with the subcategories of Tate motives, Artin motives, and Artin--Tate motives, with certain choices of fields and coefficients. With regards to the whole of $\DM_\gm(k)$, the largest known collection of points of this spectrum comes from so-called \emph{isotropic realisations}
\[\begin{tikzcd}
    \psi_{E,p}:\DM(k) \arrow[r] & \DM(\widetilde{E}/\widetilde{E};\mathbb{F}_p)
\end{tikzcd}\]
to the categories of \emph{isotropic motives} $\DM(\widetilde{E}/\widetilde{E};\mathbb{F}_p)$. Here $E/k$ is a field extension, $p \in \Z_{>0}$ is a prime, and $\widetilde{E}$ denotes the \emph{flexible closure} of the field $E$---see \Cref{remark:flexibledefinition} for the definition of a flexible field and \Cref{remark:flexibleclosureconservative} for the flexible closure of a field. This is achieved by the computation that the zero ideal $(0)$ is prime in these (geometric) isotropic categories whenever $k$ is flexible, so defines a point of $\Spc(\DM_\gm(k/k;\mathbb{F}_p))$, whence the pullbacks $a_{p,E} \coloneqq \psi_{E,p}^{-1}(0)$ define points of the spectrum of $\DM_\gm(k)$ (\cite{Vishik:MoravaKTheory}). We are then posed with the natural quest to further our understanding of the tt-geometry of these isotropic localisations in order to aid in the mission of illuminating the spectrum of Voevodsky's category of motives: this is precisely the aim of this paper.

In particular, we investigate the subcategory of \emph{isotropic Tate motives} $\DTM_\gm(k/k;\F_2)$, as well as the subcategories of \emph{isotropic Artin motives} $\DAM_\gm(k/k;\F_2)$ and \emph{isotropic Artin--Tate motives} $\DATM_\gm(k/k;\F_2)$. We refer to the next section for details on the definitions of these categories. Heuristically, $\DTM_\gm(k/k;\F_2)$ can be viewed as the `cellular' part of the category of isotropic motives, whereas $\DAM_\gm(k/k;\F_2)$ can be thought of as the `$0$-dimensional' part (and then $\DATM_\gm(k/k;\F_2)$ is the smallest tt-subcategory of $\DM_\gm(k/k;\F_2)$ containing both).

\begin{hypothesis}\label{hyp:flexiblechar0}
    Throughout this paper, unless stated otherwise, when working with isotropic categories we will always assume $k$ is flexible and of characteristic zero.
\end{hypothesis}

Let us say a few words on the choice of these categories. The advantage of restricting to the category generated by Tates with $\F_2$-coefficients is that we understand every morphism between the generators of this category: if we change to odd $p$ or consider different isotropic motives, we have very little knowledge on their morphism groups. That is to say, we understand the isotopic motivic cohomology algebra $H_{iso}^{*,*'}(\Spec(k);\F_2)$ of the point $\Spec(k)$, but have few other computations at our disposal. Thanks to results of \cite{Tanania:IsotropicFundamentalGroups}, we immediately find the isotropic motivic cohomology groups of the generators of $\DAM_\gm(k/k;\F_2)$ and $\DATM_\gm(k/k;\F_2)$ are also computable and coincide with those of the trivial field extension.  In particular, we know that all non-isomorphisms between any given generators in these categories are $\otimes$-nilpotent. In fact, more is true; they are $\otimes$-nilpotent in a strong sense: the morphism \emph{groups} vanish upon taking high enough $\otimes$-powers, not just the morphisms themselves. This strong nilpotency property turns out to be the key in proving the main theorems of this paper:
\begin{theorem}[\Cref{cor:spectrumofDTM}]
    Let $k$ be a flexible field. Then the Balmer spectrum of isotropic Tate motives $\Spc(\DTM_\gm(k/k;\F_2))$ is a singleton.
\end{theorem}
\begin{theorem}[\Cref{theorem:spectraofDAM/DATM}]
    Let $k$ be a flexible field. The spectra of isotropic Artin motives and isotropic Artin--Tate motives, $\Spc(\DAM_\gm(k/k;\F_2))$ and $\Spc(\DATM_\gm(k/k;\F_2))$, are also singletons.
\end{theorem}
Indeed, the proofs reduce to demonstrating the Tate objects $\{T(i)[j]\}_{i,j \in \Z} \subseteq \DTM_\gm(k/k;\F_2)$ (respectively the Artin--Tate objects $\{M(E)(i)[j]\}_{E/k,i,j} \subseteq \DATM_\gm(k/k;\F_2)$), form what we call a \emph{strongly (left) unitarily orthogonal} set of generators, the proof of which not only relies on the aforementioned strong nilpotency of morphisms between these objects, but also considers the bidegrees in which the nonzero morphisms can live: that they live in discrete bidegrees with exponential growth is another key ingredient of the proof.

After this, we investigate the homological spectra of these categories: we are able to completely describe these spaces, and in particular we verify Balmer's \emph{Nerves of Steel} conjecture holds for these categories. The homological spectrum of a tt-category is another topological space one associates to a tt-category, denoted $\Spc^h(\mathcal{K})$. Its points are \emph{maximal Serre $\otimes$-ideals} of the (finitely-presented) module category $\mathsf{mod}-\mathcal{K}$. In \cite{Balmer:ResidueFields}, Balmer constructs a canonical continuous map
\[\begin{tikzcd}
    \phi:\Spc^h(\mathcal{K}) \arrow[r] & \Spc(\mathcal{K})
\end{tikzcd}\]
which he proves is surjective for $\mathcal{K}$ rigid; the aforementioned conjecture stated that this comparison map should be a bijection in these cases. In this generality, this conjecture been disproven by recent work of Barthel--Hyslop--Ramzi \cite{Barthel:GeometricPoints}; despite this, we confirm it holds for all three of our categories $\DTM_\gm(k/k;\F_2)$, $\DAM_\gm(k/k;\F_2)$, $\DATM_\gm(k/k;\F_2)$, and we explicitly describe the homological primes:

\begin{theorem}[\Cref{cor:NoSholds}]
    Let $k$ be a flexible field. The \emph{Nerves of Steel} conjecture holds for $\DTM_\gm(k/k;\F_2)$. That is, the homological spectrum $\Spc^h(\DTM_\gm(k/k;\F_2))$ consists of a single homological prime. It is given by the maximal Serre $\otimes$-ideal $\ker(\widehat{t})^{fp} \subseteq \mathsf{mod}-\DTM_\gm(k/k;\F_2)$, where $t$ is the weight complex functor on $\DTM_\gm(k/k;\F_2)$ and $\widehat{t}$ denotes the induced map on the module categories. The same holds for the categories $\DAM_\gm(k/k;\F_2),\DATM_\gm(k/k;\F_2)$.
\end{theorem}

These results follow from an analysis of the `big' versions of the categories in question. Using the results developed during that section, we are also able to easily deduce that the big categories $\DTM(k/k;\F_2)$, $\DATM(k/k;\F_2)$ are \emph{not} stratified in the sense of \cite{Barthel:Stratification}:

\begin{theorem}[\Cref{prop:stratification}]
    Let $k$ be a flexible field. Neither of the categories $\DTM(k/k;\F_2)$, $\DATM(k/k;\F_2)$ are stratified.
\end{theorem}

The proof is essentially an immediate consequence of \Cref{prop:tisnotconservative} which states that the kernel of the weight complex functor on these categories is nontrivial, i.e. the functor fails to be conservative. Moreover, we can find infinitely many localising ideals in these categories---see \Cref{remark:localisingideals}. We show however that the functor \emph{is} conservative on $\DAM(k/k;\F_2)$---\Cref{prop:tisconservativeonArtins}.

The final section of this paper concerns generation properties of our categories of motives. We can ask: what are the `dimensions' of our categories? Of course, one then needs to decide what is meant by dimension. \Cref{cor:spectrumofDTM} tells us the Krull dimension of $\DTM_\gm(k/k;\F_2)$ is $0$ (that is, the Krull dimension of its Balmer spectrum is $0$), and it is a natural question to ask what the Rouquier dimension of this category is. The Rouquier dimension of a (not necessarily tensor) triangulated category $\mathcal{K}$ is the minimum $n$ for which there exists an object $a \in \mathcal{K}$ such that $\ideal{a}^\triangle_{n+1}=\mathcal{K}$, and is defined to be $\infty$ if no such $a,n$ exist. Here $\ideal{a}^\triangle_m$, for a positive integer $m$, is defined (recursively) to be the smallest full subcategory of $\mathcal{K}$ generated by $a$ closed under finite sums, summands and $m-1$ many extensions (so that $\ideal{a}^\triangle_1$ is the thick additive subcategory generated by $a$). We show the Rouquier dimension of $\DTM_\gm(k/k;\F_2)$ is infinite; moreover, we show that the category is not classically finitely generated (meaning there is no finite set---equivalently, single object---that generates the category as a triangulated category):

\begin{theorem}[\Cref{prop:DTMnotfinitelygenerated}]
    Let $k$ be a flexible field. The category $\DTM_\gm(k/k;\F_2)$ is not classically finitely generated. That is, for every object $X \in \DTM_\gm(k/k;\F_2)$ the smallest triangulated subcategory generated by $X$ is a proper subcategory. The same holds for $\DATM_\gm(k/k;\F_2)$.
\end{theorem}

Furthermore, we see that the above result is not specific to the local categories $\DM_\gm(k/k;\F_2)$. Indeed, we show that the category $\DM_\gm(k)$ (not assuming $k$ flexible) is also not classically finitely generated; although the result could easily be well-known, we include the result anyway as it may be of interest to some readers---see \Cref{prop:DMnotfinitelygenerated}.

\subsubsection*{Acknowledgements}
I am indebted to my supervisor Alexander Vishik for his continuous support and invaluable discussions, including comments on earlier drafts of this paper, that helped to make this paper possible. I would also like to thank Jean Paul Schemeil for his many stimulating questions. I am grateful to Martin Gallauer for useful conversations, in particular for pointing out that the symmetric monoidal statement within \Cref{prop:existenceoft} should follow from the construction of \cite[Rmk.~3.6]{Sosnilo:WeightStructures} using $\operatorname{Ind}$-completions, as well as providing me with a proof of \Cref{lemma:homotopycolimits}. I thank the anonymous referee for helpful comments and suggestions which helped improve the quality of this paper. This work was supported by the EPSRC Studentship [EP/W524402/1] Project Reference~2883661. For the purpose of open access, the author(s) has applied a Creative Commons Attribution (CC BY) licence to any Author Accepted Manuscript version arising. This is the Author Accepted Manuscript version of the following article: Sparks, F. (2026), On the tensor-triangular geometry of isotropic motives. J. London Math. Soc., 114: e70752., which has been published in final form at https://doi.org/10.1112/jlms.70752.
\section{Preliminaries}
These following subsections are dedicated to briefly introducing the main concepts and definitions that concern this paper. The well-acquainted reader may safely skip to the next section, or if they prefer they can use this section as a reference if an unfamiliar concept arises amidst reading the latter sections.
\subsection*{Notation} In this paper our triangulated categories' shift/suspension functors will always be denoted $(-)[1]$, i.e. we are adopting (co)homological notation, as our applications are for categories for which this is more appropriate. $\mathcal{T}$ will usually be used to denote a `big' tt-category (i.e. a rigidly-compactly generated tt-category), whereas $\mathcal{K}$ will be reserved for a `small' one, typically the rigid-compact part of a big one, i.e. $\mathcal{K}=\mathcal{T}^c$ for some $\mathcal{T}$ as detailed. $\mathcal{A}$ will be used for abelian categories, which will typically find themselves semisimple. $\Hom_\mathcal{C}(-,-)$ will denote the morphism groups between objects in $\mathcal{C}$ (all our categories will be additive): we will freely and frequently drop the subscript $\mathcal{C}$ when it is clear from context which category we are working in, and when ambiguity is improbable.

\subsection{Motives and isotropisation}
There are various constructions of categories of motives, but the one we care about in this paper is Voevodsky's (tensor-)triangulated category $\DM(k;R)$ over a field $k$ with coefficients in some commutative ring $R$. To define this category, we must first consider the category $\DM^\mathsf{eff}(k;R)$ of \emph{effective} motives, defined as the $\A^1$-localisation of the derived category of Nisnevich sheaves with transfers with values in $R$-modules, see \cite{Voevodsky:TriangulatedMotives} and \cite{Cisinski:TriangulatedMixedMotives}. We denote the motivic functor by $M:\Sch_k \rightarrow \DM^\mathsf{eff}(k;R)$, where $\Sch_k$ is the category of schemes of finite type over $\Spec(k)$, which associates to such a scheme $X$ its (effective) motive $M(X)$. We denote by $T$ the motive of $\Spec(k)$, which is the $\otimes$-unit of the monoidal structure, and we may refer to it as the \emph{trivial Tate motive}.
\begin{example}
    The motive $M(\PP^1)$ of the projective line decomposes as $T \oplus T(1)[2]$; in other words, we define the motive $T(1)[2]$ to be the object $\cone(T \rightarrow M(\PP^1))$ where the map $T \rightarrow M(\PP^1)$ is induced by the inclusion of a $k$-rational point of $\PP^1$. The direct sum decomposition is a consequence of a splitting of the structure morphism $\pi:\PP^1 \rightarrow \Spec(k)$:
    \[\begin{tikzcd}
        \PP^1 \arrow[r, "\pi"] & \arrow[l, dashrightarrow, bend right=50, swap, "s"] \Spec(k) \arrow[loop right, "\id"]
    \end{tikzcd}\]
\end{example} We define the category $\DM(k;R)$ to be the category obtained from $\DM^\mathsf{eff}(k;R)$ by formally inverting $T(1)[2]$. This contains its effective version as a full subcategory (\cite[Thm.~4.3.1]{Voevodsky:TriangulatedMotives}), and by abuse of notation we also denote by $M$ the motivic functor
\[\begin{tikzcd}
    M:\Sch_k \arrow[r] & \DM^\mathsf{eff}(k) \arrow[r] & \DM(k).
\end{tikzcd}\]
$\DM(k;R)$ is rigidly-compactly generated whenever the exponential characteristic of $k$ is invertible in $R$ (see \cite[§3]{Gallauer:TateSpectrum} and the references therein: \cite{Cisinski:TriangulatedMixedMotives}, \cite{Cisinski:IntegralMixedMotives}, \cite{Kelly:ldhdescent}), and its full subcategory of compact objects $\DM_\gm(k;R)$ is referred to as the category of \emph{geometric motives} (indeed, this category is the full thick triangulated subcategory of $\DM(k;R)$ generated by the motives $M(X)(n) \coloneqq M(X) \otimes T(n)$ for smooth varieties $X$ over $k$ and integers $n \in \Z$).

For a given prime $p>0$ and field $k$, the category of \textbf{isotropic motives} over $k$, as introduced by Vishik \cite{Vishik:IsotropicMotives} following ideas of Bachmann \cite{Bachmann:Picard}, is defined as the Verdier quotient of $\DM(k;\F_p)$ where we quotient out by the localising subcategory generated by the motives of \emph{$p$-anisotropic varieties}: these are varieties $X$ over $k$ for which the degrees of all closed points $x \in X$ are divisible by $p$. It is denoted $\DM(k/k;\F_p)$, and its compact part is denoted $\DM_\gm(k/k;\F_p)$ (which is still generated by the motives of smooth varieties, their Tate twists, and the summands of such objects).

\begin{remark}
    Intuitively, the category of isotropic motives should be viewed as taking our category of motives and trying to make it closer in complexity to the topological counterpart. Indeed, spaces $X$ with no rational points are an algebro-geometric phenomenon, and don't exist in topology. So, it seems natural to annihilate the motives of such `non-topological' spaces to get closer to the world of topology. Of course, the category of topological motives with $\F_p$-coefficients is the bounded derived category of finite-dimensional $\F_p$-vector spaces, $\D^b(\F_p)$, which is as simple as can be.
\end{remark}

\begin{remark}\label{remark:flexibledefinition}
    At this point it is worth noting it is crucial what our field $k$ is when we perform this isotropisation. If $k=\overline{k}$ is algebraically closed then there are no $p$-anisotropic varieties for any $p$ (every closed point is rational, i.e. has degree $1$), so $\DM(k/k;\F_p)=\DM(k;\F_p)$ in this case. So, we must choose a particularly nice class of fields to work with. It turns out that the correct fields to work over are so-called \textbf{flexible fields}: these are fields $k$ for which there is some other field $k_0$ and an isomorphism $k=k_0(t_1,t_2,\ldots)$, where $t_i$ are transcendental variables. In words, this means we say a field $k$ is flexible if it is a purely transcendental extension of some other field with infinite transcendence degree.
\end{remark}
\begin{remark}\label{remark:flexibleclosureconservative}
    Thanks to \cite[Rmk.~1.2]{Vishik:IsotropicMotives}, we know base-changing from a field $k$ to its flexible closure $\widetilde{k} \coloneqq k(t_0,t_1,\ldots)$ is conservative on categories of motives, so we are comforted upon passage to these fields.
\end{remark}
\begin{remark}\label{remark:smashingisotropic}
    As in \cite{Vishik:IsotropicMotives}, an equivalent formulation is the view the category $\DM(k/k;\F_p)$ as the localisation $\Upsilon_{k/k} \otimes \DM(k;\F_p)$, where $\Upsilon_{k/k}$ is the complementary $\otimes$-projector of the motive of the Čech simplicial scheme $\mathcal{X}_\mathbf{Q}$ associated to the disjoint union $\mathbf{Q}$ of all $p$-anisotropic varieties over $k$. That is, $\Upsilon_{k/k}$ is defined via the distinguished triangle
    \[\begin{tikzcd}
        \mathcal{X}_\mathbf{Q} \arrow[r] & T \arrow[r] & \Upsilon_{k/k} \arrow[r] & \mathcal{X}_\mathbf{Q}[1].
    \end{tikzcd}\]
    That $\Upsilon_{k/k}$ is an idempotent follows from the fact that $\mathcal{X}_\mathbf{Q}$ is (or more precisely, that the canonical map $\mathcal{X}_\mathbf{Q} \rightarrow T$ becomes an isomorphism after tensoring with $\mathcal{X}_\mathbf{Q}$).
\end{remark}
\begin{remark}
    In the category $\DM(k;R)$ (resp. $\DM(k/k;\F_p)$) we refer to the localising subcategory generated by the Tate twists $T(i)[j]$ for $i,j \in \Z$ as the subcategory of \textbf{Tate motives} (resp. \textbf{isotropic Tate motives}), and denote it $\DTM(k;R)$ (resp. $\DTM(k/k;\F_p)$). We refer to the localising subcategory of $\DM(k;R)$ (resp. $\DM(k/k;\F_p)$) generated by the motives of $0$-dimensional varieties (that is, the motives of finite separable field extensions $E/k$ and their tensor products) as the subcategory of \textbf{Artin motives} (resp. \textbf{isotropic Artin motives}), and denote it $\DAM(k;R)$ (resp. $\DAM(k/k;\F_p)$). Finally, by the category of \textbf{Artin--Tate motives} (resp. \textbf{isotropic Artin--Tate motives}) we mean the localising subcategory generated by Artin motives and their Tate twists, and denote it $\DATM(k;R)$ (resp. $\DATM(k/k;\F_p)$). Clearly the latter contains the former two as full subcategories. In all cases, an additional subscript $(-)_\gm$ denotes the compact part of the category, which we may also refer to as the geometric part.
\end{remark}

There are many results which illuminate these isotropic categories' simplicity when one works over a flexible field. One of the most important one for us is the following, which is a collation of Corollary~5.4, Proposition~5.7 and Lemma~5.9 of \cite{Vishik:MoravaKTheory}.

\begin{theorem}\label{thmref:heartofgeometricisotropicmotives}
    Let $k$ be a flexible field. Then there is a bounded weight structure on $\DM_\gm(k/k;\F_p)$ induced from the Chow weight structure on $\DM_\gm(k;\F_p)$. Its heart is canonically equivalent to the category of numerical Chow motives, $\Chow_\mathsf{Num}(k;\F_p)$. Moreover, this latter category is semisimple abelian (it is abelian, and every short exact sequence splits).
\end{theorem}

Another key result that we will use in this paper is the following, which is \cite[Thm.~3.7]{Vishik:IsotropicMotives}.

\begin{theorem}\label{thmref:isotropicmotiviccohofpoint}
    Let $k$ be a flexible field. Then the isotropic motivic cohomology algebra of $\Spec(k)$ is given by the exterior algebra
    \[H_{iso}^{*,*'}(\Spec(k);\F_2) \simeq \Lambda_{\F_2}(r_0,r_1,\ldots)\]
    where the $r_i$ are the duals of the Milnor operations, concentrated in bidegree $(1-2^i)[1-2^{i+1}]$.
\end{theorem}

Indeed, the preceding theorem is the main reason why it has proven beneficial to restrict ourselves to the category of isotropic (Artin--)Tate motives with $\F_2$-coefficients: this is the only situation where we know all the morphisms between all the generators of the category (at present).

\subsection{Weight structures}
Introduced by Bondarko in \cite{Bondarko:WeightStructures}, a weight structure on a triangulated category permits one to filter objects by simpler ones, similar to $t$-structures.
\begin{definition}\label{def:weightstructure}
    Let $\mathcal{C}$ be an idempotent complete triangulated category. A \textbf{weight structure} on $\mathcal{C}$ is a pair of full subcategories $w=(\mathcal{C}^{\leq 0},\mathcal{C}^{\geq 0})$ satisfying
    \begin{itemize}
        \item[(1)] $\mathcal{C}^{\leq 0},\mathcal{C}^{\geq 0}$ are idempotent complete in $\mathcal{C}$;
        \item[(2)] $\mathcal{C}^{\leq 0} \subseteq \mathcal{C}^{\leq 1}$ and $\mathcal{C}^{\geq 1} \subseteq \mathcal{C}^{\geq 0}$, where we define $\mathcal{C}^{\leq n} \coloneqq \mathcal{C}^{\leq 0}[-n], \mathcal{C}^{\geq n} \coloneqq \mathcal{C}^{\geq 0}[-n]$ for each $n \in \Z$;
        \item[(3)] $\mathcal{C}^{\geq 1} \perp \mathcal{C}^{\leq 0}$, i.e. for every $X \in \mathcal{C}^{\geq 1},Y \in \mathcal{C}^{\leq 0}$, we have $\Hom_\mathcal{C}(X,Y)=0$;
        \item[(4)] For every $X \in \mathcal{C}$ there exists a distinguished triangle
        \[A \rightarrow X \rightarrow B \rightarrow A[1]\]
        where $A \in \mathcal{C}^{\geq 1},B \in \mathcal{C}^{\leq 0}$.
    \end{itemize}
    We define the \textbf{heart} of the weight structure to be the full subcategory $\mathcal{C}^\heartsuit \coloneqq \mathcal{C}^{\leq 0} \cap \mathcal{C}^{\geq 0}$. We say a weight structure $w$ is \textbf{bounded} if every $X \in \mathcal{C}$ is in some $\mathcal{C}^{a \leq w \leq b} \coloneqq \mathcal{C}^{\geq a} \cap \mathcal{C}^{\leq b}$ for some integers $a \leq b$.
\end{definition}
\begin{definition}
    Suppose $\mathcal{C}$ is an idempotent complete tt-category. We shall call a weight structure $w=(\mathcal{C}^{\leq 0},\mathcal{C}^{\geq 0})$ on $\mathcal{C}$ \textbf{$\otimes$-compatible} if $\unit \in \mathcal{C}^\heartsuit$, $\mathcal{C}^{\leq 0} \otimes \mathcal{C}^{\leq 0} \subseteq \mathcal{C}^{\leq 0}$, and $\mathcal{C}^{\geq 0} \otimes \mathcal{C}^{\geq 0} \supseteq \mathcal{C}^{\geq 0}$.
\end{definition}
\begin{remark}
    Note that a $\otimes$-compatible weight structure on a tt-category $\mathcal{C}$ induces a symmetric monoidal structure on its heart $\mathcal{C}^\heartsuit$.
\end{remark}
\begin{example}
    Consider the triangulated category $\K(A)$ of chain complexes up to homotopy of an idempotent complete additive category $A$. Given a (co)chain complex $M^\bullet$, we can consider the \emph{`silly truncation'}
    \[\begin{tikzcd}
        w_{\leq n}:M^\bullet \arrow[r, mapsto] & w_{\leq n}M^\bullet
    \end{tikzcd}\]
    which sends $M^\bullet$ to the complex whose $m^{th}$ term is given by
    \[w_{\leq n}M^\bullet \coloneqq \begin{cases}
        M^m, \ m \leq n; \\
        0, \ m>n;
    \end{cases}\]
    and with the same differentials as $M^\bullet$ in those degrees in which we are permitted to write them down. Similarly, we may truncate from below, and these truncations assemble into a distinguished triangle in $\K(A)$:
    \[\begin{tikzcd}
        w_{\geq n}M^\bullet \arrow[r] & M^\bullet \arrow[r] & w_{\leq n-1}M^\bullet \arrow[r] & w_{\geq n}M^\bullet[1].
    \end{tikzcd}\]
    Verifying the other conditions of \Cref{def:weightstructure}, we see the following full subcategories define a weight structure on $\K(A)$:
    \[\K(A)^{\leq 0} \coloneqq \{M^\bullet \in \K(A) \mid M^\bullet \text{ is homotopic to some } N^\bullet \text{ with } N^m=0 \text{ for } m>0\};\]
    \[\K(A)^{\geq 0} \coloneqq \{M^\bullet \in \K(A) \mid M^\bullet \text{ is homotopic to some } N^\bullet \text{ with } N^m=0 \text{ for } m<0\}.\]
    If moreover $A$ is symmetric monoidal, then this weight structure is $\otimes$-compatible.
\end{example}

In general, there is no canonical choice of weight decomposition of a given object $X \in \mathcal{K}$. This means we find ourselves in the predicament where we are forced to make a \emph{choice} of the weight filtration on our objects in any given situation. That being said, we can construct differentials between the graded pieces $X_n$ of some object $X$, and we would like to say that this construction is functorial in some sense. If we try this in general, the resulting functor---the so-called \emph{weak weight complex functor}---lands in a quotient $\K_\mathfrak{w}(\mathcal{K}^\heartsuit)$ of the homotopy category of chain complexes $\K(\mathcal{K}^\heartsuit)$:
\[\begin{tikzcd}
    t_\mathfrak{w}:\mathcal{K} \arrow[r] & \K_\mathfrak{w}(\mathcal{K}^\heartsuit).
\end{tikzcd}\]
Unfortunately, the latter category fails to be triangulated in general. We refer to \cite{Bondarko:WeightStructures} for the details. Despite these shortcomings, in cases where the category $\mathcal{K}$ has a suitable enhancement, the above functor $t_\mathfrak{w}$ can be upgraded to a \emph{`strong'} weight complex functor
\[\begin{tikzcd}
    t:\mathcal{K} \arrow[r] & \K(\mathcal{K}^\heartsuit)
\end{tikzcd}\]
which is triangulated, and moreover symmetric monoidal whenever the weight structure is $\otimes$-compatible (\cite{Sosnilo:WeightStructures} gives the existence as a triangulated functor when $\mathcal{K}$ admits a stable $\infty$-categorical enhancement; \cite{Aoki:WeightFunctorIsTensor} promotes this functor to a symmetric monoidal one when $\mathcal{K}$ further admits a stably symmetric monoidal $\infty$-categorical enhancement).

Note that if we have a bounded weight structure $w$ on a triangulated category $\mathcal{C}$, then it permits us to express any given $X \in \mathcal{C}$ as a successive extension of finitely many (possibly zero) $X_i \in \mathcal{C}^\heartsuit[-i]$ of weight $i$ for $a \leq  i \leq b$ when $X \in \mathcal{C}^{a \leq w \leq b}$.
When we consider Voevodsky's triangulated category of geometric motives $\DM_\gm(k;R)$ (for simplicity, assume $k$ is perfect and the exponential characteristic of $k$ is invertible in $R$), there is a bounded $\otimes$-compatible weight structure $w_{Chow}$ called the \textbf{Chow weight structure}: its heart consists of the category $\Chow(k;R)$ of Chow motives with $R$-coefficients, that is, the motives of smooth projective varieties (and their summands), where morphisms are given by $R$-linear correspondences (i.e. $R$-linear combinations of algebraic cycles on the product of the varieties in question modulo rational equivalence). We refer to \cite{Bondarko:WeightStructures} and \cite{Bondarko:ResolutionofSingularities} for details on this weight structure.

\subsection{Tensor-triangular geometry}
Tensor-triangular geometry is concerned with studying triangulated categories $\mathcal{K}$ geometrically through their \textbf{Balmer spectra} $\Spc(\mathcal{K})$. It is defined as follows (see \cite{Balmer:tt-geometry} for details):

\begin{definition}
    Let $\mathcal{K}$ be a tensor-triangulated category. A \textbf{tt-ideal} is a full subcategory $\mathcal{I} \subseteq \mathcal{K}$ containing the zero object, which is moreover triangulated, thick, and closed under tensoring with arbitrary objects of $\mathcal{K}$.
    
    Explicitly, we require
    \begin{enumerate}
        \item $0 \in \mathcal{I}$;
        \item if $a \rightarrow b \rightarrow c \rightarrow a[1]$ is a distinguished triangle in $\mathcal{K}$, then if any two of $a,b,c$ are in $\mathcal{I}$, then so is the third;
        \item we require $\mathcal{I}$ to be closed under retracts: $a \oplus b \in \mathcal{I} \Rightarrow a,b \in \mathcal{I}$;
        \item if $x \in \mathcal{I}$ and $a \in \mathcal{K}$ then $a \otimes x \in \mathcal{I}$.
    \end{enumerate}
    We call a proper tt-ideal $\mathcal{I}$ \textbf{prime} if for all $x,y \in \mathcal{K}$, $x \otimes y \in \mathcal{I}$ implies $x \in \mathcal{I}$ or $y \in \mathcal{I}$.
\end{definition}

For a subset $\mathcal{S} \subseteq \mathcal{K}$, we write $\ideal{\mathcal{S}}$ for the tt-ideal generated by $\mathcal{S}$. As is customary, we drop parentheses when convenient: for instance $\ideal{x} \coloneqq \ideal{\{x\}}$, $\ideal{x_1,x_2,\ldots,x_n} \coloneqq \ideal{\{x_1,x_2,\ldots,x_n\}}$ and so on.

\begin{definition}
    Let $\mathcal{K}$ be an essentially small tensor-triangulated category. We define the \textbf{Balmer spectrum} $\Spc(\mathcal{K})$ as follows: as a set, it consists of the prime tt-ideals $\mathcal{P} \subseteq \mathcal{K}$. The topology, called the \textbf{Zariski topology}, is given by declaring the closed sets to be
    \[Z(S) \coloneqq \{\mathcal{P} \in \Spc(\mathcal{K}) \mid \mathcal{P} \cap S = \emptyset\}\]
    for all collections of objects $S \subseteq \mathcal{K}$. It has a basis of closed sets given by the \textbf{support} of objects $a \in \mathcal{K}$:
    \[\supp(a) \coloneqq \{\mathcal{P} \in \Spc(\mathcal{K}) \mid a \notin \mathcal{P}\}.\]
\end{definition}

One of the key motivating ideas behind studying the tt-geometry of triangulated categories is the fact that the topological space $\Spc(\mathcal{K})$ completely determines the set of \emph{radical thick tt-ideals} of $\mathcal{K}$ (if $\mathcal{K}$ is rigid, then these are precisely the tt-ideals of $\mathcal{K}$: any such is automatically radical). It does so through the description of the \emph{Thomason subsets}: these are subsets $Y \subseteq \Spc(\mathcal{K})$ that can be written as a union $Y = \cup_{i \in I}U_i$ of closed subsets $U_i \subseteq \Spc(\mathcal{K})$ each of which has quasi-compact complement; in particular, there is a bijection between the set of Thomason subsets of $\Spc(\mathcal{K})$ and the set of radical thick tt-ideals of $\mathcal{K}$, given by sending an ideal $\mathcal{I}$ to its support:
\[\mathcal{I} \mapsto \supp(\mathcal{I}) \coloneqq \bigcup_{a \in \mathcal{I}}\supp(a).\]
As a corollary, knowledge of the space $\Spc(\mathcal{K})$ allows one to classify objects of $\mathcal{K}$ up to `tt-equivalence,' that is, whenever objects $a,b \in \mathcal{K}$ generate the same (radical) tt-ideal.

\begin{remark}
    Akin to the Zariski spectra of rings, the Balmer spectrum is (contravariantly) functorial with respect to tt-functors between tt-categories: if $F:\mathcal{K} \rightarrow \mathcal{L}$ is such, then there is an induced map on the categories' spectra
    \[\Spc(F):\Spc(\mathcal{L}) \rightarrow \Spc(\mathcal{K}); \ \mathcal{P} \mapsto F^{-1}(\mathcal{P}).\]
\end{remark}

\begin{example}\label{ex:spectrumofring}(\cite{Hopkins:GlobalMethods}, \cite{Neeman:Chromatic}, \cite{Thomason:Classification}, \cite[Thm.~6.3]{Balmer:tt-geometry}, \cite[Thm.~54]{Balmer:Survey})
    Let $R$ be a commutative ring. Then consider the bounded homotopy category of (co)chain complexes of finitely-generated projective $R$-modules: $\K^b(\mathsf{proj}_R)$. This is a (rigid) tt-category in the obvious way, and we have a homeomorphism of topological spaces
    \[\Spc(\K^b(\mathsf{proj}_R)) \simeq \Spec(R)\]
    from its Balmer spectrum to the Zariski spectrum of $R$. The isomorphism is given by sending a prime $\mathfrak{p} \in \Spec(R)$ to the kernel of localisation at said prime:
    \[\begin{tikzcd}
        \mathcal{P}(\mathfrak{p}) \coloneqq \ker(-\otimes_R R_\mathfrak{p}:\K^b(\mathsf{proj}_R) \arrow[r] & \K^b(\mathsf{proj}_{R_\mathfrak{p}})).
    \end{tikzcd}\]
    In particular, we see that two objects $M,N \in \K^b(\mathsf{proj}_R)$ generate the same tt-ideal iff $\supp(M)=\supp(N)$ iff $M$ and $N$ vanish under the same localisations at primes $\mathfrak{p} \in \Spec(R)$.
    
    The above is a specialisation of the case where $X$ is a quasi-compact quasi-separated scheme: the Balmer spectrum of the category of perfect complexes over such $X$ enjoys a homeomorphism $\Spc(\D^\mathsf{perf}(X)) \simeq X$.
\end{example}
\begin{remark}
    A particular instance of the preceding example tells us the Balmer spectrum of the (compact part of the) topological category of motives is the spectrum of the integers:
    \[\Spc(\D(\Z)^c) \simeq \Spec(\Z)\]
    whilst the $\F_p$-coefficients version of the category gives rise to a singleton spectrum.
\end{remark}

As well as the aforementioned reasons for wanting to understand the Balmer spectra of categories, one particular motivation for studying the tt-geometry of $\DM_\gm(k;R)$ is to understand conservative realisations. Indeed, a conservative family of functors $F_i:\DM_\gm(k;R) \rightarrow \mathcal{K}_i$ is more-or-less equivalent to surjectivity of the induced map on spectra on closed points (see \cite{Balmer:Conservative} and \cite{Barthel:Surjectivity}), and one of the large goals in the theory of motives is to produce conservative families of functors to easier, more manageable categories (and there are big conjectures relating to such problems).

On top of the Balmer spectrum, there is another topological space one associates to tt-categories, called the \textbf{homological spectrum}, denoted $\Spc^h(\mathcal{K})$. We refer the reader to \cite{Balmer:ResidueFields} and the references therein for the intricacies of the homological spectra of tt-categories; nevertheless we recall the relevant basic details here. When considering the homological spectrum, we often assume $\mathcal{K}$ is the `small' part of a `big' tt-category $\mathcal{T}$: the latter should admit arbitrary coproducts and be rigidly-compactly generated by the essentially small tt-category $\mathcal{K}$ (although this rigidly-compactly generated assumption isn't necessary to define the homological spectrum). We consider the module category $\mathsf{Mod}-\mathcal{K}$, defined to be the category of additive contravariant functors from $\mathcal{K}$ to the category of abelian groups $\mathsf{Ab}$:
\[\mathsf{Mod}-\mathcal{K} \coloneqq \operatorname{Fun}^{\oplus}(\mathcal{K}^{op},\mathsf{Ab}).\]
We then consider the subcategory $\mathsf{mod}-\mathcal{K}$ of finitely-presented objects; both are abelian categories. We can consider the restricted Yoneda functor
\[h:\mathcal{T} \rightarrow \mathsf{Mod}-\mathcal{K}\]
which sends an object $x \in \mathcal{T}$ to the restriction $\Hom_\mathcal{T}(-,x)\vert_{\mathcal{K}}$ of the functor represented by $x$ to the subcategory $\mathcal{K}$. The restriction of $h$ to $\mathcal{K}$ factors through $\mathsf{mod}-\mathcal{K}$, and this is the universal homological functor out of $\mathcal{K}$. It is common to denote $h$ by $\widehat{(-)}$ (on both objects and morphisms). We call a full subcategory $\mathcal{B} \subseteq \mathcal{A}$ of an abelian category $\mathcal{A}$ a \emph{Serre subcategory} if for every short exact sequence
\[0 \rightarrow x \rightarrow y \rightarrow z \rightarrow 0\]
in $\mathcal{A}$, we have that $y \in \mathcal{B}$ iff both $x,z \in \mathcal{B}$. In the symmetric monoidal setting, we call $\mathcal{B}$ a Serre $\otimes$-ideal if it is Serre and moreover closed under taking tensor products with arbitrary objects of $\mathcal{A}$. The category $\mathsf{Mod}-\mathcal{K}$ can be made symmetric monoidal: the tensor product is given by Day convolution---this is the (unique) symmetric monoidal structure on $\mathsf{Mod}-\mathcal{K}$ which is colimit-preserving in each variable whilst making $h\vert_\mathcal{K}$ tensor. Moreover, the image of each $x \in \mathcal{T}$ under $h$ is $\otimes$-flat in $\mathsf{Mod}-\mathcal{K}$ (see \cite{Balmer:SmashingFrame}).

\begin{definition}
    Keeping the above notation, we define the \textbf{homological spectrum} $\Spc^h(\mathcal{K})$ of $\mathcal{K}$ to be the following space. As a set, it consists of the maximal Serre $\otimes$-ideals $\mathcal{B} \subseteq \mathsf{mod}-\mathcal{K}$. The topology is the one generated by the closed subsets of the form
    \[\supp^h(a) \coloneqq \{\mathcal{B} \in \Spc^h(\mathcal{K}) \mid \widehat{a} \notin \mathcal{B}\}\]
    for $a \in \mathcal{K}$.
\end{definition}

The quotient categories $\mathsf{Mod}-\mathcal{K}/\langle \mathcal{B} \rangle$ are referred to as \textbf{homological residue fields} (where, here, $\ideal{\mathcal{B}}$ denotes the localising subcategory of $\mathsf{Mod}-\mathcal{K}$ generated by $\mathcal{B}$). There is a canonical map from the homological spectrum of $\mathcal{K}$ to its Balmer spectrum, since the former defines a support data for $\mathcal{K}$. A key result of Balmer \cite{Balmer:ResidueFields} is that this \emph{comparison map}
\[\phi:\Spc^h(\mathcal{K}) \longrightarrow \Spc(\mathcal{K}); \ \mathcal{B} \mapsto h^{-1}(\mathcal{B})\]
is surjective whenever $\mathcal{K}$ is rigid. Consequently, it is clear that the complexity of the homological spectrum bounds from above the complexity of the usual spectrum of a rigid tt-category. Moreover, in \textit{loc. cit.} Balmer proposed what came to be known as the \emph{Nerves of Steel} conjecture, stating that the map $\phi$ should be a bijection when $\mathcal{K}$ is rigid: this conjecture has been proven false in the recent work of Barthel--Hyslop--Ramzi \cite{Barthel:GeometricPoints}. Nevertheless, it is of course an interesting question as to when the comparison map is bijective, and we shall say the \emph{Nerves of Steel} conjecture holds for $\mathcal{K}$ whenever the comparison map $\phi$ above is indeed bijective.
\section{Balmer spectra of tt-categories possessing sufficiently orthogonal generators}
Throughout this section, when considering isotropic motives, $p$ will denote any prime, and our field $k$ will always be flexible and of characteristic zero (\Cref{hyp:flexiblechar0}), unless stated otherwise. Let $\mathcal{K}$ be an essentially small rigid tt-category, and assume $\mathcal{K}$ is equipped with a $\otimes$-compatible bounded weight structure $w$. Suppose $\mathcal{K}$ is generated by a set of objects $\mathcal{S}=\{x_i\}_{i \in I}$ each of which lie in a single weight (that is, for each $i \in \Z$, $x_i \in \mathcal{K}^\heartsuit[-n_i]$ lies in weight $n_i$ for some $n_i \in \Z$), and suppose every $x \in \mathcal{K}^\heartsuit[-n]$ is a finite coproduct of (possibly various) $x_j$ for which $n_j=n$. Suppose further that we have $\ideal{x_i}=\mathcal{K}$ for every $i \in I$.

\begin{definition}\label{def:strongnilpotency}
    With the notation as above, we shall call a (finite) subset $\{x_j\}_{j \in J} \subseteq \mathcal{S}$ \textbf{strongly left unitarily orthogonal} if there is some $n \in \Z_{>0}$ for which
    \[\Hom(x_{\psi(1)} \otimes \cdots \otimes x_{\psi(n)},\unit)=0\]
    for all functions $\psi:\{1,\ldots,n\} \rightarrow J$.
\end{definition}
\begin{theorem}\label{theorem:nilpotentspectrum}
    Let $\mathcal{K}$ be as above. Suppose for every $a \in \mathcal{K}$, the weight complex of $a$ can be chosen to have zero differentials, and that every object of $\mathcal{K}^{<0}$ has a weight decomposition consisting of a strongly left unitarily orthogonal set of the generators (this will happen if for instance every (finite) subset $\{x_j\}_{j \in J} \subseteq \mathcal{S} \cap \mathcal{K}^{<0}$ is strongly left unitarily orthogonal).
    Then the Balmer spectrum $\Spc(\mathcal{K})$ is the singleton $\{(0)\}$.
\end{theorem}
\begin{proof}
    We will show that for any nonzero object $a \in \mathcal{K}$, the tt-ideal generated by $a$ is the entire category. This will prove the spectrum is a singleton (the only prime will be $(0)$).

    The existence of (bounded) weight decompositions permits us to express $a$ as follows:
    \[\begin{tikzcd}
        a_m \arrow[r] & a \arrow[r] & a_{\leq m-1} \arrow[r, "f"] & a_m[1]
    \end{tikzcd}\]
    where $m$ is the `highest weight' of $a$ (so that $a \in \mathcal{K}^{\leq m}$ and $a_m \in \mathcal{K}^\heartsuit[-m]$). Note that $\ideal{a}$ coinciding with $\mathcal{K}$ is equivalent to proving the quotient category $\mathcal{K}/\ideal{a}$ is trivial. In other words, we need to show that forcing the morphism $f$ to become an isomorphism kills our category. Thus, it is enough to show that $f$ is $\otimes$-nilpotent (or \emph{smash nilpotent}). For, in this case, we have that (the image of) $f$ is both $\otimes$-nilpotent and an isomorphism in $\mathcal{K}/\ideal{a}$, which forces the (image of the) target $a_m[1]$ of $f$ to be zero in $\mathcal{K}/\ideal{a}$, which in turn forces the quotient category to be zero. This last fact follows by noting $\ideal{a_m}=\mathcal{K}$ (as $a_m$ is the coproduct of some $x_i$). Explicitly, we would have, for some $n$, the morphism in $\mathcal{K}/\ideal{a}$
    \[0=f^{\otimes n}:(a_{\leq m-1})^{\otimes n} \rightarrow a_m^{\otimes n}[n]\]
    so that $a_m^{\otimes n}=0 \in \mathcal{K}/\ideal{a}$ (since $a=\cone(f)[-1]$, annihilating $a$ forces $f$, and so $f^{\otimes n}$, to become an isomorphism in this quotient). Then we have (looking in the quotient category $\mathcal{K}/\ideal{a}$)
    \[\mathcal{K}/\ideal{a}=\ideal{a_m}=\ideal{a_m^{\otimes n}}=0,\]
    so we conclude $\ideal{a}=\mathcal{K}$ (the second equality here follows from rigidity).
    
    By assumption, we can choose our weight filtration on $a$ to be such that the differentials $d$ between the weight graded pieces are zero. Thus, via considering the following diagram (more precisely, by considering the associated long exact sequence)
    \[\begin{tikzcd}
        a_{m-1} \arrow[r] \arrow[rd, swap, "d=0"] & a_{\leq m-1} \arrow[r] \arrow[d, "f"] & a_{\leq m-2} \arrow[r] \arrow[dl, dashrightarrow] &  a_{m-1}[1] \\
        & a_m[1]
    \end{tikzcd}\]
    we obtain that $f$ factors through $a_{\leq m-2}$. In particular it suffices to prove that any such morphism
    \[\begin{tikzcd}
        a_{\leq m-2} \arrow[r] & a_m[1]
    \end{tikzcd}\]
    is $\otimes$-nilpotent. Equivalently (by shifting the above morphism) we require every map of the form
    \[\begin{tikzcd}
        b \arrow[r] & x
    \end{tikzcd}\]
    to be $\otimes$-nilpotent, where $b \in \mathcal{K}^{<0}$ lives in negative degree with respect to the weight structure and $x \in \mathcal{K}^\heartsuit$. By rigidity, we may consider the dual map
    \[\begin{tikzcd}
        b \otimes x^\vee \arrow[r] & \unit
    \end{tikzcd}\]
    and so (as $b \otimes x^\vee \in \mathcal{K}^{<0}$) without loss of generality we may assume $x=\unit$.
    
    By assumption, we can choose a strongly left unitarily orthogonal weight decomposition of $b$, say filtered by $b_0,\ldots,b_l$ living in negative degrees (with respect to $w$). Thus, for each $n>0$, $b^{\otimes n}$ will be an extension of combinations of these:
    \[\{b_{\psi(1)} \otimes \cdots \otimes b_{\psi(n)} \mid \psi:\{1,\ldots,n\} \rightarrow \{0,\ldots,l\}\}\]
    where $\psi$ ranges over all set-theoretic functions from $\{1,\ldots,n\}$ to $\{0,\ldots,l\}$. Each $b_i$ itself is a coproduct of some finite number of $x_j$'s. By definition, there exists an $n$ for which the morphism groups $\Hom(x_{\psi(1)} \otimes \cdots \otimes x_{\psi(n)},\unit)$ are zero. Distributivity of the coproduct gives us the groups $\Hom(b_{\psi(1)} \otimes \cdots \otimes b_{\psi(n)},\unit)$ are also zero, whence the group $\Hom(b^{\otimes n},\unit)$ is zero too (since this group is a iterated extension of the aforementioned morphism groups). This concludes the proof.
\end{proof}
\begin{remark}
    Another way of seeing that the cone of a $\otimes$-nilpotent morphism $f:c \rightarrow d$---with $d$ a unit---generates the whole category is by inductively using the inclusion
    \[\ideal{\cone(f^{\otimes 2})} \subseteq \ideal{\cone(f)}\]
    to get
    \[\ideal{\cone(f)} \subseteq \ideal{c,d}=\ideal{c^{\otimes n},d^{\otimes n}}=\ideal{\cone(f^{\otimes n})} \subseteq \ideal{\cone(f)}\]
    (where $n$ is such that $f^{\otimes n}=0$), where we have invoked rigidity to achieve the equality $\ideal{c,d}=\ideal{c^{\otimes n},d^{\otimes n}}$. Thus, the tt-ideal generated by $\cone(f)$ coincides with the tt-ideal generated by both $c$ and $d$.
\end{remark}
\begin{remark}
    We could have equally considered the lowest weight of $a$ and the morphism groups $\Hom(\unit,x_{\psi(1)} \otimes \cdots \otimes x_{\psi(n)})$ instead. Indeed, one can define the notion of \emph{strongly right unitarily orthogonal} using these groups in place of $\Hom(x_{\psi(1)} \otimes \cdots \otimes x_{\psi(n)},\unit)$ in \Cref{def:strongnilpotency}, and one can obtain a dual version of the preceding theorem in this way.
\end{remark}

\begin{theorem}\label{theorem:DTMsatisfiesthecondition}
    $\DTM_\gm(k/k;\F_2)$ satisfies the conditions of \Cref{theorem:nilpotentspectrum}.
\end{theorem}
\begin{proof}
    We must verify the conditions of \Cref{theorem:nilpotentspectrum}.
    The preamble is immediate, taking the generating set $\mathcal{S}$ to be the set $\{T(i)[j] \in \DM_\gm(k/k;\F_2) \mid (i,j) \in \Z^2\}$. We can always guarantee the weight complex of any object consists of zero differentials, since the heart of our weight structure is semisimple abelian (it is equivalent to the category of finite-dimensional graded $\F_2$-vector spaces). It remains to verify the strongly left unitarily orthogonal condition. So let $X \in \DM_\gm(k/k;\F_2)^{<0}$ be an isotropic Tate motive living in negative degree with respect to the (isotropic) Chow weight structure. Note that any $T(i)[j] \in \DTM_\gm(k/k;\F_2)^\heartsuit$ is of the form $T(i)[2i]$, so $T(i)[j]$ lives in $\DTM_\gm(k/k;\F_2)^\heartsuit[-n]$ iff $j-2i=-n$.
    
    $X$ is then an extension of finitely many isotropic Tate motives living above the \emph{Chow line}, $Chow_0$. That is, $X$ is an extension of finitely many Tates $x_i=T(a_i)[b_i]$ (say, $i=1,\ldots,m$) for which $b_i-2a_i>0$, and we must show that any such finite set is strongly left unitarily orthogonal. Then in the figure below, the `round' $(a)$ coordinates of these Tates for which $X$ is an extension of are concentrated in the horizontal axis $[-M,M]$ for some positive integer $M$. Then the `round' $(a)$ coordinates of the Tates $T_\psi \coloneqq x_{\psi(1)} \otimes \cdots \otimes x_{\psi(n)}$ will live in $[-Mn,Mn]$ for every $\psi:\{1,\ldots,n\} \rightarrow \{1,\ldots,m\}$. Let us denote the line $b=2a+l$, the \emph{$l^{th}$-Chow line}, as $Chow_l$. The conditions $b_i-2a_i>0$ forces each graded piece to lie at least distance $1$ away from the Chow line $Chow_0$, i.e. lie at least on or above $Chow_1$, so the coordinates of $T_\psi$ lie at least on or above the line $Chow_n$. The isotropic motivic cohomology of the point is discrete: it is given by the exterior algebra $\Lambda_{\F_2}(r_i \mid i=0,1,\ldots)$ with generators $r_i$ living in bidegrees $(-2^i+1)[-2^{i+1}+1]$ (corresponding to the duals of the Milnor operations)---\Cref{thmref:isotropicmotiviccohofpoint}. Thus, the `limiting bound' coordinates for these nonzero cohomology groups are given by the points corresponding to $r_0r_1 \cdots r_l$ for each $l$, which live in bidegree
    \[\sum_{i=0}^l (-2^i+1)[-2^{i+1}+1] = (-2^{l+1}+l+2)[-2^{l+2}+l+3].\]
    By `limiting bound' we mean the coordinates of every other nonzero cohomology class lies in the convex hull of these points. To verify strong left unitarily orthogonality we must check the Tates $T_\psi$ corresponding to functions $\psi:\{1,\ldots,n\} \rightarrow \{1,\ldots,m\}$ fail to reach the origin via a combination of these cohomology classes $r_i$ for sufficiently large $n$. In other words, we need to show that the sum of the coordinates $(a_i,b_i)$ associated to the Tates $T(a_i)[b_i]$ cannot coincide with the negative of the coordinates of the nonzero isotropic motivic cohomology groups of the point $\Spec(k)$ as long as we take arbitrarily many terms of the sum. Thus it suffices to show the sums of $n$-times these coordinates are disjoint from the region corresponding to the convex hull of the points $(2^{l+1}-l-2)[2^{l+2}-l-3]$ for $n \gg 0$. But note that this region reaches the $l^{th}$-Chow line $Chow_l$ after $l$ steps (the point corresponding to $r_0r_1 \cdots r_{l-1}$ lies on the line $Chow_l$). At this elevation, the $(a)$ coordinate of such points are asymptotically $2^{l+1}$ to the right, whereas the coordinates of the $T_\psi$ are bounded by $[-Mn,Mn]$, so that the green, dashed piecewise linear curve in the figure below grows exponentially to the right to reach any given Chow line, whereas the coordinates of our $T_\psi$ of interest only grow linearly. This means there are no morphisms from each $T_\psi$ to the $\otimes$-unit $T$ for sufficiently large $n$, so we have proven the strong left unitarily orthogonality condition we are after. This concludes the proof.
    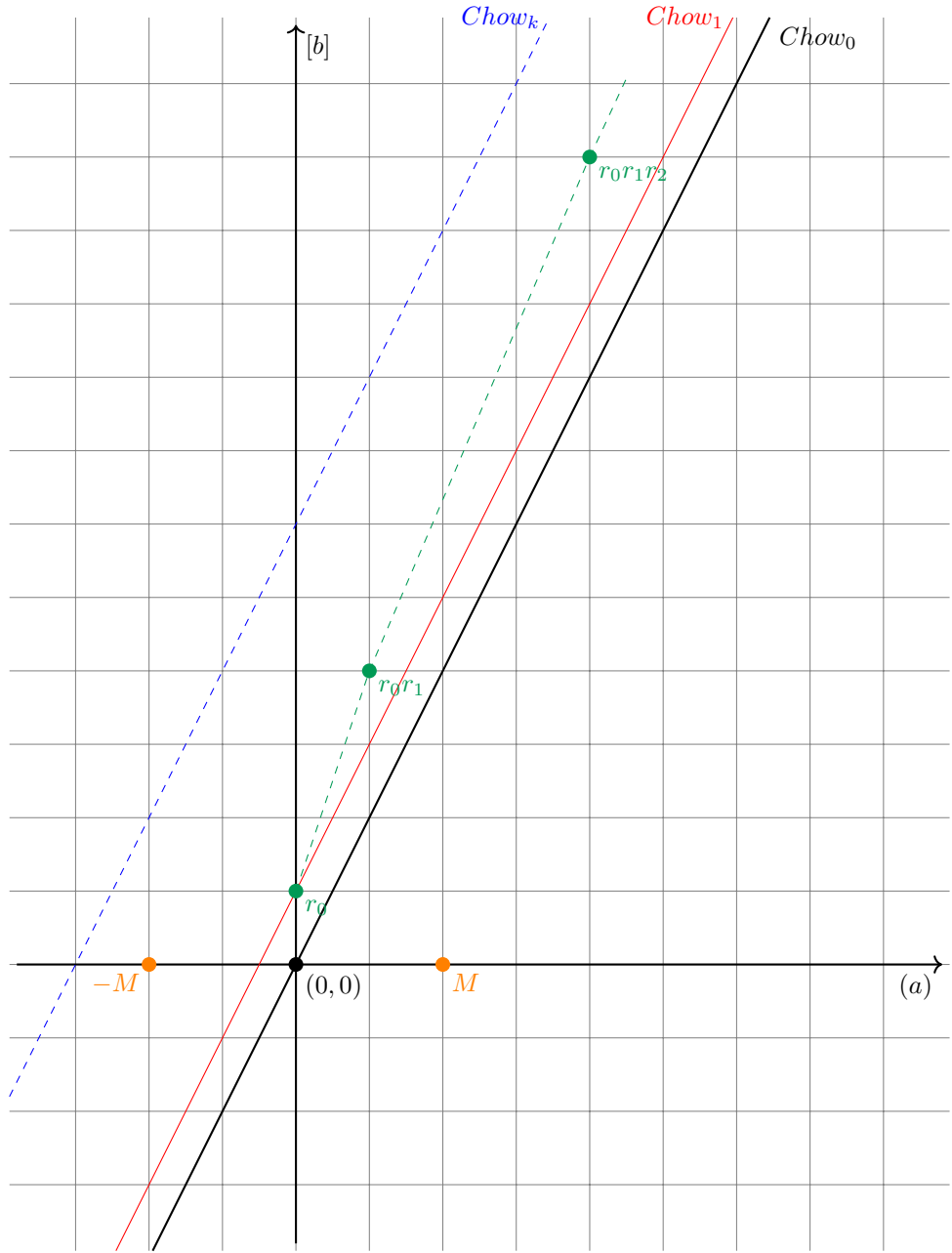
\begin{figure}[h!]
    \begin{center}
        \begin{tikzpicture}
            \draw[step=1cm,gray,very thin] (-3.9,-3.9) grid (8.9,12.9);
            \draw[thick,->] (-3.8,0) -- (8.8,0) node[anchor=north east] {$(a)$};
            \draw[thick,->] (0,-3.8) -- (0,12.8) node[anchor=north west] {$[b]$};
            \draw[thick] (-1.95,-3.9) -- (6.45,12.9) node[anchor=north west] {$Chow_0$};
            \draw[thin, red] (-2.45,-3.9) -- (5.95,12.9) node[anchor=east] {$Chow_1$};
            \draw[thin, blue, dashed] (-3.9,-1.8) -- (3.45,12.9) node[anchor=east] {$Chow_n$};
            \fill[black] (0,0) circle (0.1cm) node[anchor=north west] {$(0,0)$};
            \fill[orange] (2,0) circle (0.1cm) node[anchor=north west] {$M$};
            \fill[orange] (-2,0) circle (0.1cm) node[anchor=north east] {$-M$};
            \fill[ForestGreen] (0,1) circle (0.1cm) node[anchor=north west] {$r_0$};
            \fill[ForestGreen] (1,4) circle (0.1cm) node[anchor=north west] {$r_0r_1$};
            \fill[ForestGreen] (4,11) circle (0.1cm) node[anchor=north west] {$r_0r_1r_2$};
            \draw[thin, ForestGreen, dashed] (0,1) -- (1,4);
            \draw[thin, ForestGreen, dashed] (1,4) -- (4,11);
            \draw[thin, ForestGreen, dashed] (4,11) -- (4.5,11+15/14);
        \end{tikzpicture}
    \end{center}
    \caption{The coordinates of the associated graded pieces of $X^{\otimes n }$ lie above $Chow_n$ and have $(a)$-coordinate bounded between $-Mn$ and $Mn$ for some fixed $M$. The negation of the coordinates of the isotropic motivic cohomology groups of the point $\Spec(k)$ are bounded above by the green dashed line. For large enough $n$, these regions are disjoint.}
    \end{figure}
\end{proof}
\begin{cor}\label{cor:spectrumofDTM}
    The Balmer spectrum $\Spc(\DTM_\gm(k/k;\F_2))$ is a singleton.
\end{cor}

\begin{theorem}\label{theorem:spectraofDAM/DATM}
    All of the above holds if we replace $\DTM_\gm(k/k;\F_2)$ with isotropic Artin motives $\DAM_\gm(k/k;\F_2)$ or isotropic Artin--Tate motives $\DATM_\gm(k/k;\F_2)$. That is, they satisfy the conditions of \Cref{theorem:nilpotentspectrum}. In particular, each of these categories have a singleton spectrum.
\end{theorem}
\begin{proof}
    Let $\mathcal{K}$ be one of the categories $\DAM_\gm(k/k;\F_2)$ or $\DATM_\gm(k/k;\F_2)$. In either case, the weight structure on $\DM_\gm(k/k;\F_2)$ restricts to $\mathcal{K}$ whose heart consists of pure isotropic Artin motives for $\DAM_\gm(k/k;\F_2)$ and pure isotropic Artin motives with a pure shift $(-)(i)[2i]$ for $\DATM_\gm(k/k;\F_2)$: both are semisimple abelian categories. We must calculate the isotropic motivic cohomology of $\Spec(E)$ for some field extension $E$ over $k$, which is given by
    \[H_{iso}^{*,*'}(\Spec(E);\F_2) \simeq H_{iso}^{*,*'}(\Spec(k);\F_2) \otimes BP_{iso}^{*,*'}(\Spec(E))\]
    by \cite[Rmk.~4.7]{Tanania:IsotropicFundamentalGroups}. Now \cite[Prop.~5.8]{Tanania:IsotropicFundamentalGroups} gives that the isotropic motivic mod $2$ $BP$-theory of a field $F$ is given by its isotropic Milnor $K$-theory concentrated along the diagonal
    \[BP_{iso}^{*,*'}(\Spec(F)) \simeq k^M_*(F/k).\]
    So, let $E/k$ be a finite field extension. If $[E:k]$ is even, then $M(\Spec(E))$ vanishes in $\mathcal{K}$ (alternatively we could invoke \cite[Rmk.~5.3]{Tanania:IsotropicFundamentalGroups}). If $[E:k]$ is odd, then $k^M_*(E/k)$ is $\F_2$ concentrated in degree $0$ (\cite[Prop.~6.6]{Tanania:IsotropicFundamentalGroups}), so its isotropic motivic cohomology coincides with that of the trivial Tate motive. We just need to verify that the tensor product of two motives of field extensions also has its isotropic motivic cohomology concentrated in the same degrees. This follows from noting that if $E,F$ are two field extensions of $k$, then their tensor product $E \otimes_k F$ decomposes as a product of field extensions over $k$ whenever one of $E$ or $F$ is separable over $k$.
    
    Now the proof of \Cref{theorem:DTMsatisfiesthecondition} carries over almost verbatim: we just need to note that any nonzero object $X$ of the heart (that is, some summand of $M(E)$ for some field extension $E$ over $k$) generates the whole category. Indeed, the heart is semisimple abelian, so the $\otimes$-unit $\unit=T=M(\Spec(k))$ is a summand of $X \otimes X^\vee$:
    \[0 \neq \Hom(X,X) \simeq \Hom(T,X \otimes X^\vee)\]
    (here we use that $T$ is indecomposable) so that $T \in \ideal{X}$, and we are done.
\end{proof}
\begin{remark}
    Let $\mathcal{K}$ be any one of the categories $\DTM(k/k;\F_p)$, $\DAM(k/k;\F_p)$, $\DATM(k/k;\F_p)$. The endomorphism ring $\End_\mathcal{K}(\unit)$ of the unit in $\mathcal{K}$ is given by $\F_p$, so we find that each of the Balmer spectra calculated above are isomorphic to $\Spec(\F_2)$ as locally ringed spaces.
\end{remark}
\begin{remark}
    In view of the categories $\DTM_\gm(k/k;\F_2)$, $\DAM_\gm(k/k;\F_2)$ being full subcategories of $\DATM_\gm(k/k;\F_2)$, we could have also deduced the former two categories have singleton spectra by only considering the calculation for $\DATM(k/k;\F_2)$ and then invoking \cite[Cor.~1.8]{Balmer:Conservative}.
\end{remark}

It is worth noting what goes wrong when we try to apply to same proof of \Cref{theorem:DTMsatisfiesthecondition} to the entire category of isotropic motives $\DM_\gm(k/k;\F_2)$. As already mentioned in the proof of the above, the isotropic motivic cohomology of a smooth variety $X$ is given by
\[H_{iso}^{*,*'}(X;\F_2) \simeq H_{iso}^{*,*'}(\Spec(k);\F_2) \otimes BP_{iso}^{*,*'}(X),\]
where $BP_{iso}^{*,*'}(X)$ denotes the isotropic motivic mod $2$ $BP$-theory of $X$. The proof of \Cref{theorem:DTMsatisfiesthecondition} breaks down in this generality since we have the potential existence of nontrivial isotropic motivic cohomology classes (that is, coming from the $BP$-theory, rather than the cohomology of $\Spec(k)$) living below the Chow line $Chow_0$, arising from positive-dimensional smooth projective varieties $X$. For instance, a nontrivial class $r \in H_{iso}^{1,1}(X;\F_2)$ can be seen to obstruct the line of reasoning of the proof: $X(-1)[-1]$ `lives' on the Chow line $Chow_1$ and has a morphism to the origin, corresponding to $r$. We currently do not have any computations of nontrivial isotropic $BP$-groups for positive dimensional varieties, however it is plausible that the existence of such could give rise to nonzero primes in $\DM_\gm(k/k;\F_2)$. For instance, a nontrivial class $f \in H_{iso}^{*,*'}(X;\F_2)$ may be non-$\otimes$-nilpotent. Then, since for any $\mathcal{K}$ and morphism $f$, $\mathcal{I}_f \coloneqq \{a \in \mathcal{K} \mid \exists n>0, \ f^{\otimes n} \otimes \id_a=0\}$ is a tt-ideal by \cite[Prop.~2.12]{Balmer:Spectra}, this could provide a straightforward method to verify new points of the spectrum $\Spc(\DM_\gm(k/k;\F_2))$. Of course, one would need $f$ to be $\otimes$-nilpotent on some nonzero object $a$, and then this would imply $a$ is contained in some nonzero prime ideal: $f$ failing to be $\otimes$-nilpotent ensures the tt-ideal $\mathcal{I}_f$ is proper (for the unit fails to reside there), so $f$ being $\otimes$-nilpotent on $a \neq 0$ guarantees $a$ is contained in some nontrivial prime $\mathcal{P} \supseteq \mathcal{I}_f \supsetneq (0)$. Thus, it is imperative to understand the isotropic $BP$-theory of positive dimensional varieties to shed light on the tensor-triangular geometry of $\DM_\gm(k/k;\F_2)$.

\section{The homological spectrum}
We begin this section by proving a series of results on `big' tt-categories that are rigidly-compactly generated by their compact part, the latter of which possesses a weight structure $w$ with semisimple abelian heart. All of the results then specialise to the categories of `big' isotropic Tate, Artin and Artin--Tate motives, which are vital to understanding the homological spectrum of the `small' categories. Throughout this section, as in the preceding one, $p$ will denote any prime, and our field $k$ will always be flexible and of characteristic zero when considering isotropic categories (\Cref{hyp:flexiblechar0}), unless stated otherwise. $\mathcal{K}$ will always denote an essentially small rigid tt-category which rigidly-compactly generates a big tt-category $\mathcal{T}$. We also assume $\mathcal{K}$ has a bounded weight structure $w$ whose heart $\mathcal{A}$ is a Krull--Schmidt semisimple abelian category (so that in particular every object of $\mathcal{A}$ is a finite coproduct of simple ones). Recall that whenever $\mathcal{K}$ admits a stably symmetric monoidal $\infty$-categorical enhancement, the (strong) weight complex functor
\[t:\mathcal{K} \rightarrow \K^b(\mathcal{A})\]
exists by \cite{Sosnilo:WeightStructures} and is tensor whenever the weight structure is $\otimes$-compatible by \cite{Aoki:WeightFunctorIsTensor}.

\begin{prop}\label{prop:extendingweightstructures}
    We can extend the weight structure $w$ on $\mathcal{K}$ to a weight structure $w'$ on $\mathcal{T}$ whose heart $\mathcal{T}^\heartsuit$ is the coproduct completion $\mathcal{A}_\oplus$ of $\mathcal{A}$, that is, the category obtained from $\mathcal{A}$ by formally adjoining all small coproducts.
\end{prop}
\begin{proof}
    This is an immediate consequence of \cite[Thm.~4.1.2]{Bondarko:Hurewicz} applied to our situation. The explicit description of the non-positive part $\mathcal{T}^{\leq 0}$ is as the idempotent completion of the coproduct and extension closure of $\mathcal{K}^{\leq 0}$ in $\mathcal{T}$. The non-negative part is then given by
    \[\mathcal{T}^{\geq 0} \coloneqq \prescript{\perp}{}{(\mathcal{T}^{\leq -1})}.\]
    The only nontrivial thing to note is that in our situation, the \emph{smashing idempotent completion} of $\mathcal{A}$ as defined in \textit{loc. cit.} (i.e. the idempotent completion $(\mathcal{A}_\oplus)^\natural$ of $\mathcal{A}_\oplus$) coincides with the coproduct completion $\mathcal{A}_\oplus$ of $\mathcal{A}$ by \Cref{prop:semisimpleabeliancompletion} below.
\end{proof}

\begin{prop}\label{prop:semisimpleabeliancompletion}
    The coproduct completion $\mathcal{A}_\oplus$ of $\mathcal{A}$ is also a semisimple abelian category.
\end{prop}
\begin{proof}
    Let $X,Y \in \mathcal{A}_\oplus$. By definition, both $X$ and $Y$ are direct sums of objects of $\mathcal{A}$, and so we can write
    \[X \simeq \bigoplus_{\lambda \in \Lambda}\bigoplus_{i \in I_\lambda} A^{(\lambda)}_i\]
    where $\Lambda$ is an indexing set for the isomorphism classes $\{A^{(\lambda)}\}$ of the simple objects of $\mathcal{A}$. In other words, we can split $X$ up into arbitrary direct sums of pairwise isomorphic objects. Similarly for $Y$, we may write
    \[Y \simeq \bigoplus_{\lambda \in \Lambda}\bigoplus_{j \in J_\lambda}A^{(\lambda)}_j.\]
    Then by definition we have
    \[\Hom_{\mathcal{A}_\oplus}(X,Y) \simeq \prod_{\lambda_1 \in \Lambda}\prod_{i \in I_{\lambda_1}}\bigoplus_{\lambda_2 \in \Lambda}\bigoplus_{j \in J_{\lambda_2}}\Hom_\mathcal{A}(A^{(\lambda_1)}_i,A^{(\lambda_2)}_j).\]
    We have $\Hom_\mathcal{A}(A,B)=0$ for non-isomorphic simple objects $A,B \in \mathcal{A}$, and for each simple object $A^{(\lambda)}$ the endomorphism ring $\End_\mathcal{A}(A^{(\lambda)})$ is a division ring (every nontrivial endomorphism is an isomorphism by simplicity). Thus, we obtain an equivalence of categories
    \[\mathcal{A}_\oplus \simeq \bigoplus_{\lambda \in \Lambda}\mathsf{Mod}-\End_\mathcal{A}(A^{(\lambda)})\]
    between $\mathcal{A}_\oplus$ and the products of the categories of (right) modules over the endomorphism rings of the simple objects $A^{(\lambda)}$ where $\lambda \in \Lambda$ ranges over the isomorphism classes of simple objects. Since the ring $\End_\mathcal{A}(A^{(\lambda)})$ is division, its category of modules is semisimple abelian. It then remains to verify that a direct sum of semisimple abelian categories is once more semisimple abelian: any morphism $f:X \rightarrow Y$ in such a category splits into its components $f_{\lambda}:X_\lambda \rightarrow Y_\lambda$, and each of these decompose as an isomorphism and a zero morphism, so $f$ does too. This concludes the proof.
\end{proof}
\begin{remark}\label{remark:coproductcompletionisindcompletion}
    Note that, in the case $\mathcal{B}$ is a semisimple stable $\infty$-category in which every object is a finite coproduct of simple objects, the coproduct completion $\mathcal{B}_\oplus$ of $\mathcal{B}$ coincides with its $\operatorname{Ind}$-completion. Indeed, in this case $\mathcal{B}_\oplus$ is cocomplete and compactly generated by $\mathcal{B}$.
\end{remark}

\begin{cor}\label{cor:nozerodivisors}
    Suppose moreover $\mathcal{A}$ is a rigid symmetric monoidal category with indecomposable $\otimes$-unit $\unit$. Then for any $x,y \in \K(\mathcal{A}_\oplus)$ we have $x \otimes y \simeq 0$ only if $x \simeq 0$ or $y \simeq 0$, and the only nonzero tt-ideal of $\K(\mathcal{A}_\oplus)$ is the entire category. The same holds for $\K^b(\mathcal{A})$: the spectrum $\Spc(\K^b(\mathcal{A}))=\{(0)\}$ is a singleton.
\end{cor}
\begin{proof}
    Observe that for any $a \in \mathcal{A}$, its dual exists and therefore $a \otimes a^\vee$ contains the unit $\unit$ as a direct summand (by adjunction and semisimplicity). We note $\mathcal{A}_\oplus$ has no $\otimes$-zero divisors, since $\mathcal{A}$ has none: for any nonzero $x,y \in \mathcal{A}$, $x \otimes y$ cannot be zero since the previous comment informs us we have a split monomorphism
    \[\unit \simeq \unit \otimes \unit \hookrightarrow x \otimes y \otimes x^\vee \otimes y^\vee.\]
    The preceding proposition implies the total cohomology functor is conservative (any complex can be represented by one with zero differentials). Thus, $x \otimes y \simeq 0$ implies $x \simeq 0$ or $y \simeq 0$ for any pair $x,y \in \K(\mathcal{A}_\oplus)$. Since the category $\K(\mathcal{A}_\oplus)$ is itself semisimple abelian (using \Cref{prop:semisimpleabeliancompletion} and the fact that the homotopy category of chain complexes of a semisimple category is also semisimple), any nontrivial tt-ideal contains some compact $a \in \K(\mathcal{A}_\oplus)$ which we've seen generates everything. The final assertion regarding $\K^b(\mathcal{A})$ can be seen via the same argument.
\end{proof}

\begin{prop}\label{prop:existenceoft}
    Assume further that our categories $\mathcal{K},\mathcal{T}$ have stably symmetric monoidal $\infty$-categorical enhancements, and that the weight structure on $\mathcal{K}$ is $\otimes$-compatible. Then there exists a strong weight complex functor
    \[t:\mathcal{T} \rightarrow \K(\mathcal{A}_\oplus).\]
    Moreover, it is symmetric monoidal.
\end{prop}
\begin{proof}
    The existence follows by \cite[Rmk.~3.6]{Sosnilo:WeightStructures}, since our source category $\mathcal{T}$ is compactly generated by $\mathcal{K}$ by assumption. There, we see that the global weight complex functor is defined by passing to $\operatorname{Ind}$-completions (working with the enhancements). Since the weight complex functor between the compact subcategories is symmetric monoidal by \cite{Aoki:WeightFunctorIsTensor}, we immediately deduce the same result for the big categories by taking $\operatorname{Ind}$-completions, taking into account \Cref{remark:coproductcompletionisindcompletion}.
\end{proof}

Let us consider the specific case of isotropic motives for a moment. All of the above applies to the category of isotropic motives $\DM_\gm(k/k;\F_p) \subseteq \DM(k/k;\F_p)$ to get that the big category has a weight structure whose heart is the semisimple abelian category $\Chow_{\mathsf{Num},\oplus}(k;\F_p)$ whose objects are coproducts of (compact) isotropic (which coincides with numerical) Chow motives with $\F_p$-coefficients (note $\Chow_\mathsf{Num}(k;\F_p)$ is Krull--Schmidt since the endomorphism ring of every object in this category is finite by \cite[Cor.~4.13]{Vishik:MoravaKTheory}). In the bounded case, the weight complex functor is conservative. For the unbounded version, we have the following result (cf. \cite[Rmk.~2.3.5(3)]{Bondarko:KillingWeights}, \cite[Lemma~2.4]{Ayoub:Conjectures}).

\begin{prop}\label{prop:tisnotconservative}
    Restrict to the case $p=2$. The weight complex functor
    \[\begin{tikzcd}
        t:\DM(k/k;\F_2) \arrow[r] & \K(\Chow_{\mathsf{Num},\oplus}(k;\mathbb{F}_2))
    \end{tikzcd}\]
    is not conservative.
\end{prop}
\begin{proof}
    Consider the isotropic motive
    \[\Omega \coloneqq \hocolim(T \overset{r_0}{\longrightarrow} T[-1] \overset{r_1[-1]}{\longrightarrow} \cdots) \in \DM(k/k;\F_2)\]
    given by the homotopy colimit of the duals of the Milnor operations
    \[r_i \in H_{iso}^{-2^{i+1}+1,-2^i+1}(\Spec(k);\F_2).\]
    Homotopy colimits turn into direct colimits under the inspection of compact objects (\cite[Lemma~2.8]{Neeman:BrownRep}), so $\Omega$ is nonzero since
    \[\Hom(T,\Omega) \simeq \colim_i \Hom(T,T(-2^{i+1}+i+2)[-2^{i+2}+i+3])\]
    is nonzero, for the product of any finite collection of the $r_i$ is nonzero (\Cref{thmref:isotropicmotiviccohofpoint}: the cohomology algebra is the exterior algebra generated by the $r_i$'s). However, it is evident that $t(\Omega)$ is zero:
    \[t(\Omega)=\hocolim_i t(T(-2^{i+1}+i+2)[-2^{i+2}+i+3])=0\]
    since $t(T(-2^{i+1}+i+2)[-2^{i+2}+i+3])$ is the pure motive $T\{-2^{i+1}+i+2\}$ concentrated in degree $i+1$, and there are evidently no morphisms in $\K(\Chow_{\mathsf{Num},\oplus}(k;\mathbb{F}_2))$ between two successive such complexes. So, $\Omega \in \ker(t) \setminus \{0\}$.
\end{proof}
\begin{remark}
    We believe that the above should also hold for odd $p$. Indeed, all we require is a sequence of morphisms that don't compose to give the zero morphism yet each vanish under the weight complex functor $t$.
\end{remark}
\begin{remark}\label{remark:tisnotconservativeforTates}
    The isotropic motive $\Omega$ also lives in (big) subcategories of isotropic (Artin--)Tate motives.
\end{remark}

Recall from \cite{Balmer:Ruminations} that a \textbf{tt-field} is a big tt-category $\mathcal{F}$ such that every object $x \in \mathcal{F}$ can be written as a coproduct of compact-rigid objects $x_i \in \mathcal{F}^c$ and the endofunctor $x \otimes -:\mathcal{F} \rightarrow \mathcal{F}$ is faithful (if $x \neq 0$).

\begin{lemma}\label{lemma:homotopycategoryisattfield}
    Suppose again $\mathcal{A}$ is moreover a rigid symmetric monoidal category with indecomposable $\otimes$-unit $\unit$. Then the category $\K(\mathcal{A}_\oplus)$ is a tt-field.
\end{lemma}
\begin{proof}
    The first condition is clear from semisimplicity (which holds by \Cref{prop:semisimpleabeliancompletion}). For the second condition, suppose $x \in \K(\mathcal{A}_\oplus)$ is nonzero and that $x \otimes f=0$ for some morphism $f:y \rightarrow z$. By semisimplicity, $f$ decomposes as an isomorphism and a zero morphism: write $y=w \oplus y'$, $z=w \oplus z'$ so that $f:y \rightarrow z$ is identified with the morphism $\id_w \oplus 0$. Then $x \otimes f=0$ implies $x \otimes w \simeq 0$. But there are no $\otimes$-zero divisors in our category (\Cref{cor:nozerodivisors}), so as $x$ is nonzero, we have $w \simeq 0$, i.e. $f=0$.
\end{proof}

\begin{prop}\label{prop:homologicalprimeabove0}
    Suppose $\mathcal{K},\mathcal{T}$ have stably symmetric monoidal $\infty$-categorical enhancements, that the weight structure on $\mathcal{K}$ is $\otimes$-compatible, and that the $\otimes$-unit $\unit \in \mathcal{A} \subseteq \mathcal{K}$ is indecomposable. Let $\widehat{t}:\mathsf{Mod}-\mathcal{K} \rightarrow \mathsf{Mod}-\K^b(\mathcal{A})$ be the map induced between the module categories from the weight complex functor $t$. Then $\mathcal{B}=\ker(\widehat{t})^{fp} \in \Spc^h(\mathcal{K})$ is a homological prime in the fibre over $(0) \in \Spc(\mathcal{K})$ under the continuous surjection
    \[\begin{tikzcd}
    \phi: \Spc^h(\mathcal{K}) \arrow[r, twoheadrightarrow] & \Spc(\mathcal{K})
    \end{tikzcd}\]
    of \cite{Balmer:ResidueFields}.
\end{prop}
\begin{proof}
    By \cite[Thm.~5.17]{Balmer:Ruminations} and \Cref{lemma:homotopycategoryisattfield}, $(0)$ is the unique point of $\Spc^h(\K^b(\mathcal{A}))$. Now we apply \cite[Thm.~5.10]{Balmer:BigSupport} to obtain the result.
\end{proof}

\begin{prop}\label{prop:NoSiffnilpotency}
    Suppose again $\mathcal{K},\mathcal{T}$ have stably symmetric monoidal $\infty$-categorical enhancements, that the weight structure on $\mathcal{K}$ is $\otimes$-compatible, and that the $\otimes$-unit $\unit \in \mathcal{A} \subseteq \mathcal{K}$ is indecomposable. Then $\Spc^h(\mathcal{K})$ is a singleton iff the weight complex functor
    \[t:\mathcal{K} \rightarrow \K^b(\mathcal{A})\]
    detects $\otimes$-nilpotence, i.e. $t(f)=0$ implies $f$ is $\otimes$-nilpotent.
\end{prop}
\begin{proof}
    The condition of the homological spectrum being a singleton implies $t$ detects $\otimes$-nilpotency by \Cref{prop:homologicalprimeabove0} and \cite[Thm.~1.1]{Balmer:ResidueFields}, which tells us that any morphism that is in the kernel of every homological residue field is $\otimes$-nilpotent. On the other hand, detecting $\otimes$-nilpotency implies there can be no other homological primes by \cite[Thm.~A.1]{Balmer:BigSupport}. This final claim uses the fact that, in our case, $t(f) \otimes t(g)=0$ implies $t(f)=0$ or $t(g)=0$.
\end{proof}

\begin{remark}
    Note that, in general, the weight complex functor need not detect $\otimes$-nilpotence. Indeed, there are morphisms $f:T \rightarrow T(1)[1]$ in $\DM_\gm(k)$ which are not $\otimes$-nilpotent (for example, we could take $k=\R$ and $f \in H_\mathcal{M}^{1,1}(\Spec(k)) \simeq K^M_1(\R)$ corresponding to the pure symbol $\{-1\}$), yet we obviously have $t(f)=0$ (the domain and codomain of $t(f)$ are one-term complexes concentrated in different degrees).
\end{remark}
\begin{theorem}\label{theorem:weightcomplexfunctordetectsnilpotence}
    Let $\mathcal{K}$ be as in \Cref{theorem:nilpotentspectrum}. Suppose moreover that $\mathcal{K}$ has a stably symmetric monoidal $\infty$-categorical enhancement and that the heart of the weight structure on $\mathcal{K}$ is a semisimple abelian category. Then the weight complex functor $t$ detects $\otimes$-nilpotence, that is, for any morphism $f:x \rightarrow y$ in $\mathcal{K}$, $t(f)=0$ implies $f^{\otimes n}=0$ for some $n$.
\end{theorem}
\begin{proof}
    Let $f:x \rightarrow y$ be any morphism in $\mathcal{K}$ such that $t(f)=0$. Then to show $f$ is $\otimes$-nilpotent, it is enough to show the dual map $f^\vee:x \otimes y^\vee \rightarrow \unit$ is $\otimes$-nilpotent, as in \Cref{theorem:nilpotentspectrum}. In other words, without loss of generality we may assume $y=\unit$ is the $\otimes$-unit. Now considering a weight decomposition of $x$ with zero differentials, we see that, as there are no morphisms from objects of positive weight to $\unit$, the morphism $f:x \rightarrow \unit$ factors through its truncation at weight $0$, $x_{\leq 0}$:
    \[\begin{tikzcd}
        x_{\geq 1} \arrow[r] \arrow[rd, swap, "0"] & x \arrow[r] \arrow[d, "f"] & x_{\leq 0} \arrow[r] \arrow[dl, dashrightarrow] &  x_{\geq 1}[1] \\
        & \unit
    \end{tikzcd}\]
    In particular, it sufficies to show that $f_{\leq 0}:x_{\leq 0} \rightarrow \unit$ is $\otimes$-nilpotent. But by examining the triangle
    \[\begin{tikzcd}
        x_0 \arrow[r] \arrow[rd, swap, "0"] & x_{\leq 0} \arrow[r] \arrow[d, "f_{\leq 0}"] & x_{\leq -1} \arrow[r] \arrow[dl, dashrightarrow] &  x_0[1] \\
        & \unit
    \end{tikzcd}\]
    we see that this map factors further through $x_{\leq -1}$: indeed, the map $x_0 \rightarrow \unit$ is zero by the assumption $t(f)=0$. But now we are reduced to proving any morphism $x_{\leq -1} \rightarrow \unit$ is $\otimes$-nilpotent for $x_{\leq -1}$ living in $\mathcal{K}^{<0}$, which was shown during the course of the proof of \Cref{theorem:nilpotentspectrum}.
\end{proof}
\begin{cor}\label{cor:NoSholds}
    For $\mathcal{K}$ satisfying the hypotheses of \Cref{prop:NoSiffnilpotency} and \Cref{theorem:weightcomplexfunctordetectsnilpotence}, the comparison map from the homological spectrum to the Balmer spectrum is bijective: the homological spectrum $\Spc^h(\mathcal{K})$ is a singleton (equal to $\{\ker(\widehat{t})^{fp}\}$).
\end{cor}

\begin{cor}\label{cor:NoSforTates}
    Balmer's \emph{Nerves of Steel} conjecture holds for $\DTM_\gm(k/k;\F_2)$. Its homological spectrum consists of the homological prime $\ker(\widehat{t})^{fp}$.
\end{cor}
\begin{proof}
    Immediate from \Cref{cor:NoSholds} combined with \Cref{theorem:DTMsatisfiesthecondition}.
\end{proof}
\begin{cor}
    The preceding result holds for both $\DAM_\gm(k/k;\F_2)$ and $\DATM_\gm(k/k;\F_2)$. That is, their homological spectra are singletons and consist of the kernel $\ker(\widehat{t})^{fp}$ of the weight complex functor.
\end{cor}
\begin{proof}
    Identical to \Cref{cor:NoSforTates}.
\end{proof}
\begin{remark}\label{remark:homologicalprimeisnonzero}
    In the case $\mathcal{K}$ is one of $\DM_\gm(k/k;\F_2)$, $\DTM_\gm(k/k;\F_2)$, $\DATM_\gm(k/k;\F_2)$, \Cref{prop:tisnotconservative}, \Cref{remark:tisnotconservativeforTates} and \Cref{cor:NoSforTates} tell us that the homological prime $\ker(\widehat{t})^{fp} \in \Spc^h(\mathcal{K})$ of \Cref{prop:homologicalprimeabove0} is nonzero. Indeed, the nonzero ideal $\ker(\widehat{t})$ is generated under filtered colimits by its finitely-presented part (\cite[Prop.~A.6]{Balmer:Ruminations}), so the homological prime in question vanishes if and only if $\ker(\widehat{t})$ vanishes. To conclude, one just needs to note the restricted Yoneda $h:\mathcal{T} \rightarrow \mathsf{Mod}-\mathcal{K}$ has trivial kernel (by virtue of compact generation) so, in particular, $h(\Omega) \in \ker(\widehat{t})$ is nonzero, where $\Omega$ is as in \Cref{prop:tisnotconservative}.
\end{remark}
\begin{remark}
    It is perhaps worth noting that despite the simplicity of $\DTM(k/k;\F_2)$, it is \emph{not} a tt-field. Indeed, for a tt-field $\mathcal{F}$, any tt-functor $F:\mathcal{F} \rightarrow \mathcal{C}$ which is conservative on compact objects and geometric (i.e. preserves coproducts) is conservative everywhere: for $0 \neq X \in \mathcal{F}$, we simply write $X=\oplus x_i$ for nonzero compact $x_i$ and then $F(X)=\oplus F(x_i) \neq 0$. Since the weight complex functor $t$ on $\DTM(k/k;\F_2)$ is not conservative (\Cref{prop:tisnotconservative}), we conclude $\DTM(k/k;\F_2)$ can't be a tt-field ($\Omega$ can't be written as the coproduct of compact objects). Moreover, for a tt-field $\mathcal{F}$ the Balmer spectrum $\Spc(\mathcal{F}^c)$ is a singleton (\cite[Prop.~5.15]{Balmer:Ruminations}), and so there are no nontrivial thick tt-ideals in $\mathcal{F}$ (as every such contains a compact object by definition of being a tt-field). Consequently, the kernel of such any tt-functor $F:\mathcal{F} \rightarrow \mathcal{G}$, to some other tt-category $\mathcal{G}$, is either $\mathcal{F}$ or zero. That is, $F$ is zero if it is not conservative.
\end{remark}
\begin{remark}
    As in \cite{Barthel:Surjectivity}, surjectivity of the induced map on homological spectra (of the subcategories of compact objects) is equivalent to detecting weak ring objects, that is, $\Spc^h(t)$ is surjective iff for any weak ring object $X$, we have $t(X)=0$ only if $X=0$. Since the \emph{Nerves of Steel} conjecture holds for our categories in question, this becomes equivalent to the induced map on the Balmer spectra of compact objects being surjective. As pointed out in \textit{loc. cit.}, this is \emph{not} equivalent to the joint conservativity of the family of functors being studied, and $\DTM_\gm(k/k;\F_2)$ gives another example of this being an insufficient criterion: that is to say, the induced map $\Spc(t):\Spc((\F_2-\mathsf{Vect}_{**})^c) \rightarrow \Spc(\DTM_\gm(k/k;\F_2))$ is surjective yet the map on the big categories $t:\DTM(k/k;\F_2) \rightarrow \F_2-\mathsf{Vect}_{**}$ fails to be conservative (and the analogous remark for $\DATM(k/k;\F_2)$ also holds).
\end{remark}

\section{Stratification}
This section deals with stratification of our categories, in the sense of \cite{Barthel:Stratification} (there are various notions of stratification for big tt-categories $\mathcal{T}$, but in this paper we only concern ourselves with stratification via the spectrum of the compacts $\Spc(\mathcal{T}^c)$). First let us recall what is meant by stratification of big tt-categories: a rigidly-compactly generated tt-category $\mathcal{T}$ with a weakly Noetherian Balmer spectrum is \emph{stratified} if it satisfies a \emph{local-to-global} principle and \emph{minimality} condition. To briefly summarise, there is an idempotent $g(\mathcal{P}) \in \mathcal{T}$ associated to each prime $\mathcal{P} \in \Spc(\mathcal{T}^c)$ (which we will recall below); the local-to-global condition asserts we have
\[\mathcal{T} = \Loc_\otimes\ideal{g(\mathcal{P}) \mid \mathcal{P} \in \Spc(\mathcal{T}^c)},\]
and the minimality condition requires $g(\mathcal{P}) \otimes \mathcal{T}$ to be a minimal localising $\otimes$-ideal.
The object $g(\mathcal{P})$ is defined via
\[g(\mathcal{P}) \coloneqq e_{Y_1} \otimes f_{Y_2}\]
where $Y_1,Y_2 \subseteq \Spc(\mathcal{T}^c)$ are Thomason subsets such that $\{\mathcal{P}\} = Y_1 \cap (\Spc(\mathcal{T}^c) \setminus Y_2)$ (and indeed the resulting object $g(\mathcal{P})$ is independent of such choices), and the objects $e_Y,f_Y$ for a given Thomason subset $Y \subseteq \Spc(\mathcal{T}^c)$ are defined via an idempotent triangle
\[\begin{tikzcd}
    e_Y \arrow[r] & \unit \arrow[r] & f_Y \arrow[r] & e_Y[1]
\end{tikzcd}\]
(that is, $e_Y \otimes f_Y \simeq 0$) with $\ker(f_Y \otimes -) = \Loc\ideal{\mathcal{T}_Y^c}$ where
\[\mathcal{T}_Y^c \coloneqq \{x \in \mathcal{T}^c \mid \supp(x) \subseteq Y\}.\]
By \cite[Thm.~A]{Barthel:Stratification}, being stratified is equivalent to having a bijective correspondence between localising $\otimes$-ideals of $\mathcal{T}$ and subsets of $\Spc(\mathcal{T}^c)$, given by taking supports.

\begin{prop}\label{prop:stratification}
    The categories $\DTM(k/k;\F_2)$, $\DATM(k/k;\F_2)$ are not stratified.
\end{prop}
\begin{proof}
    By \cite[Thm.~A]{Barthel:Stratification} mentioned above, we immediately deduce that a big tt-category $\mathcal{T}$ with $\Spc(\mathcal{T}^c)=\{0\}$ a singleton is stratified iff the only localising $\otimes$-ideals of $\mathcal{T}$ are the trivial ones (the zero ideal and $\mathcal{T}$ itself). The Balmer specturm of the compact part of the categories in the statement are singletons by \Cref{cor:spectrumofDTM} and \Cref{theorem:spectraofDAM/DATM}, however we know the big categories have nontrivial localising $\otimes$-ideals: $\ker(t)$ gives nontrivial localising $\otimes$-ideals of $\DTM(k/k;\F_2)$ and $\DATM(k/k;\F_2)$ by \Cref{prop:tisnotconservative}. So, the categories are not stratified.
\end{proof}
\begin{remark}
    As our spectrum $\Spc(\DTM_\gm(k/k;\F_2))$ is moreover Noetherian, Theorem~B of \textit{loc. cit.} tells us the local-to-global principle is satisfied. So, the minimality condition is what fails. This states that the localising $\otimes$-ideal $g(\mathcal{P}) \otimes \DTM(k/k;\F_2)$ for the (unique) prime $\mathcal{P}=(0) \in \Spc(\DTM_\gm(k/k;\F_2))$ is not minimal. This localising ideal can easily be seen to equal $\DTM(k/k;\F_2)$ itself: for $\mathcal{P}=(0)$ we see that we are forced to choose $Y_1=\Spc(\DTM_\gm(k/k;\F_2))$, $Y_2=\emptyset$. Then
    \[\mathcal{T}_{Y_1}^c=\mathcal{T}^c; \ \mathcal{T}_{Y_2}^c=(0)\]
    so that $f_{Y_1}=0=e_{Y_2}$ whilst $e_{Y_1}=T=f_{Y_2}$ (the only $f$ for which $f \otimes -$ has kernel $\mathcal{T}^c$ is $f=0$, whilst $f \otimes -$ having trivial kernel implies $e=0$ as $e \otimes f = 0$, so the distinguished triangle forces the map $\unit \rightarrow f$ to be an isomorphism), whence $g(\mathcal{P})=T \otimes T=T$ is the $\otimes$-unit.

    We have already observed that $\DTM(k/k;\F_2)$ is not a minimal localising $\otimes$-ideal. The same holds if we replace $\DTM(k/k;\F_2)$ with $\DATM(k/k;\F_2)$ everywhere. Additionally, we note the above computation renders the fact that the local-to-global principle holds in our case trivial.
\end{remark}

\begin{remark}\label{remark:localisingideals}
    One can actually show that there are infinitely many distinct localising tt-ideals in the categories $\DTM(k/k;\F_2)$ and  $\DATM(k/k;\F_2)$. To see this, for a given infinite subset $S \subseteq \N$ of natural numbers, we define the isotropic motive
    \[\Omega(S) \coloneqq \hocolim_{n \in S}(r_n)\]
    where we start from the $\otimes$-unit $T$, and we equip $S$ with the order inherited from $\N$. For example, the motive $\Omega$ of \Cref{prop:tisnotconservative} corresponds to the subset $S=\N$. It is not hard to see that every such $\Omega(S)$ is a nonzero noncompact object which lies in the kernel of the weight complex functor $t$. One can show that for subsets $S,S' \subseteq \N$, we have $\Loc_\otimes\ideal{\Omega(S)} \simeq \Loc_\otimes\ideal{\Omega(S')}$ iff $S,S'$ have finite discrepancy, that is, when they differ by only a finite set of elements. This relies on the observation that $\Omega(S) \otimes  \Omega(S')$ vanishes iff $S \cap S'$ is infinite (cf. \Cref{remark:squarestozero} below), and using the fact that (up to shifts) we can ignore finitely many elements of a given $S \subseteq \N$ (since we can truncate homotopy colimits to get isomorphic objects).
\end{remark}

It would be interesting to know whether our categories of isotropic (Artin--)Tate motives satisfy the telescope conjecture (that every \emph{smashing} $\otimes$-ideal is compactly generated). An instance of a compactly generated smashing ideal is given isotropisation itself: by \Cref{remark:smashingisotropic}, the ideal $\ideal{M(X) \mid X \ p\text{-anisotropic}} \subseteq \DM(k;\F_p)$ defining the isotropic category $\DM(k/k;\F_p)$ is smashing, and is visibly compactly generated. It is known stratified categories satisfy the telescope conjecture whenever their Balmer spectrum (of the compact objects) is generically Noetherian (\cite[Thm.~I]{Barthel:Stratification}). The proof of the above fails to disprove the telescope conjecture for $\DTM(k/k;\F_2)$ directly since we do not know if the $\otimes$-ideal $\ker(t)$ is smashing. By \cite{Balmer:Telescope} we know that smashing subcategories bijectively correspond to right idempotents $\eta:\unit \rightarrow f$ in our (big) category, up to isomorphism (it is not hard to prove no such idempotents exist in the compact part of our categories, except the trivial ones $0:\unit \rightarrow 0$ and $\id:\unit \rightarrow \unit$). Thus to understand the telescope conjecture for our category we need to understand such idempotents. One can show $\ker(t)=\ker(U(\unit) \otimes -)$, where $t \dashv U$, but we do not know if $U(\unit)$ is idempotent.

Moreover, \cite[Prop.12.7]{Barthel:Descent} applied to the category $\DAM(k/k;\F_2)$ gives that the big category of isotropic Artin motives over a flexible field with $\F_2$-coefficients is stratified iff the image of the unit $U(\unit)$ under the right adjoint to $t$ generates the entire category. Indeed, their result states that if you have a geometric tt-functor $F:\mathcal{T} \rightarrow \mathcal{S}$ between big tt-categories, then the unit $\unit_\mathcal{T}$ belonging to the localising ideal generated by $U(\unit_\mathcal{S})$ along with the identities $\Supp(F(t))=\Spc(F^c)^{-1}(\Supp(t))$ for every $t \in \mathcal{T}$ is equivalent to saying $\mathcal{T}$ is stratified alongside $\Spc(F^c)$ being surjective. As all our spectra in question are singletons, the support $\Supp(x)$ of any object $x$ (compact or not) is empty iff $x \simeq 0$. The claim now follows since our target category $\K(\mathcal{A}_\oplus)$ is evidently stratified.

The upshot of all of this is that in order to understand the stratification (or lack thereof) of $\DAM(k/k;\F_2)$, as well as the telescope conjecture for any of our isotropic categories, a better understanding of what the right adjoint to the weight complex functor looks like is essential.

Since our obstruction to stratification above arose from the kernel of the weight complex functor on the large category, let us investigate this kernel further. We begin with a lemma about morphism groups of compactly generated triangulated categories admitting stable $\infty$-categorical enhancements.

\begin{lemma}\label{lemma:homotopycolimits}
    Let $\mathcal{T}$ be a compactly generated triangulated category with a stable $\infty$-categorical enhancement $\mathcal{C}$. Then for every object $X \in \mathcal{T}$, we can write it as a filtered homotopy colimit $X \simeq \hocolim_{i \in I}(x_i)$ of compact objects $x_i \in \mathcal{T}^c$, with $I$ a filtered diagram in $\mathcal{T}$. Moreover, letting $x \in \mathcal{T}^c$ be a compact object, we have an identification of $\Hom$-sets in $\mathcal{T}$
    \[\Hom_\mathcal{T}(x,X) \simeq \colim_{i \in I}\Hom_\mathcal{T}(x,x_i).\]
\end{lemma}
\begin{proof}
    The fact that we can express such an $X$ as a filtered colimit in the $\infty$-category $\mathcal{C}$ is by definition of compact generation for $\infty$-categories ($\mathcal{C}$ is the $\operatorname{Ind}$-completion of its compact objects). We may assume the diagram $I$ is the (nerve of a) 1-category by \cite[Prop.~5.3.1.18]{Lurie:HTT}. Compact objects commute with filtered colimits (again by definition) so we have an equivalence of mapping spaces
    \[\Map_\mathcal{C}(x,X) \simeq \colim_{i \in I}\Map_\mathcal{C}(x,x_i).\]
    To recover the morphism sets in $\mathcal{T}$, we take connected components. $\pi_0$ commutes with colimits---it is left adjoint to the inclusion of (the nerve of) the category of discrete spaces---see \cite[Prop.~5.5.6.18]{Lurie:HTT}---so we reduce to working in the nerve of the category of sets. The final required observation is that colimits in the nerves of $1$-categories agree with colimits computed in the $1$-categories themselves.
\end{proof}
\begin{remark}
    In the case where the diagram is given by the poset $I=\N$, the final assertion of \Cref{lemma:homotopycolimits} holds without the need of an enhancement by \cite[Lemma~2.8]{Neeman:BrownRep} (which we used in \Cref{prop:tisnotconservative}).
\end{remark}

Now let $\mathcal{K}$ be as in \Cref{cor:NoSholds}, let $\mathcal{T}$ be the big tt-category it rigidly-compactly generates, and extend the weight complex functor to $t:\mathcal{T} \rightarrow \K(\mathcal{A}_\oplus)$ as in \Cref{prop:existenceoft}.

\begin{lemma}\label{lemma:expressionofkernel}
    Every $X \in \ker(t) \subseteq \mathcal{T}$ can be expressed as the filtered homotopy colimit
    \[X \simeq \hocolim_{i \in I}(x_i)\]
    with each $x_i \in \mathcal{K}$ and such that, for each $i \in I$, there is some morphism $f_{ij}:x_i \rightarrow x_j$ of the $I$-shaped diagram in $\mathcal{T}$ which is $\otimes$-nilpotent. Conversely, every such $X$ is in $\ker(t)$: this describes all objects of the kernel.
\end{lemma}
\begin{proof}
    Let $X \in \mathcal{T}$. Using the stable $\infty$-categorical enhancement, we may write
    \[X \simeq \hocolim_{i \in I}(x_i)\]
    for some filtered category $I$ and compact $x_i \in \mathcal{K}$.
    
    Now assume $t(X)=0$. Then we have
    \[0=t(X)=\hocolim_{i \in I}(t(x_i)) \in \K(\mathcal{A}_\oplus).\]
    Now we work in the category $\K(\mathcal{A}_\oplus)$, which is semisimple abelian (see the proof of \Cref{cor:nozerodivisors}).
    
    Fix some $i \in I$ and consider $t(x_i)$. If every morphism $t(f_{ij}):t(x_i) \rightarrow t(x_j)$ coming from the $I$-shaped diagram in $\mathcal{T}$ were nonzero, then the homotopy colimit $t(X)$ would be nonzero, for then there is a nonzero map $t(x_i) \rightarrow t(X) \in \colim_{j \in I}\Hom(t(x_i),t(x_j))$ (invoking \Cref{lemma:homotopycolimits}). This contradicts $t(X)=0$, so we must have that, for each $i \in I$, there is some $j \in I$ and a vanishing morphism $t(f_{ij}):t(x_i) \rightarrow t(x_j)$. But by \Cref{theorem:weightcomplexfunctordetectsnilpotence}, this is equivalent to each $f_i$ being $\otimes$-nilpotent.
    
    The converse is easy (the target category has no nonzero $\otimes$-nilpotence morphisms, cf.~\Cref{lemma:homotopycategoryisattfield}).
\end{proof}
\begin{lemma}\label{lemma:DAMcomposition}
    In $\mathcal{K}=\DAM_\gm(k/k;\F_2)$, the composition of $\otimes$-nilpotent morphisms
    \[\begin{tikzcd}
        x \arrow[r, "f"] & y \arrow[r, "g"] & z
    \end{tikzcd}\]
    is zero.
\end{lemma}
\begin{proof}
    Let
    \[\begin{tikzcd}
        x \arrow[r, "f"] & y \arrow[r, "g"] & z
    \end{tikzcd}\]
    be morphisms between compact objects such that both are $\otimes$-nilpotent, which is equivalent to vanishing under the weight complex functor $t$ by \Cref{prop:NoSiffnilpotency}. Then since the dual $h^\vee:\unit \rightarrow b \otimes a^\vee$ of any given map $h:a \rightarrow b$ between rigid-compacts $a,b$ is given by the composition
    \[\begin{tikzcd}
        \unit \arrow[r, "\eta"] & a \otimes a^\vee \arrow[r, "h \otimes \id_{a^\vee}"] & b \otimes a^\vee
    \end{tikzcd}\]
    we see that $(g \circ f)^\vee$ is equal to
    \[(g \circ f)^\vee = (g \otimes \id_{x^\vee}) \circ f^\vee.\]
    Thus, without loss of generality we may assume $x=\unit$ (as then we can apply the result to the composition
    \[\begin{tikzcd}
        \unit \arrow[r, "f^\vee"] & y \otimes x^\vee \arrow[r, "g \otimes \id_{x^\vee}"] & z \otimes x^\vee
    \end{tikzcd}\]
    to obtain the result).

    Fix weight decompositions of $y,z$ with zero differentials. Considering the diagram
    \[\begin{tikzcd}
        & \unit \arrow[dl, dashrightarrow] \arrow[d, "f"] \arrow[rd, "0"] &  & \\
        y_{\geq 0} \arrow[r] & y \arrow[r] & y_{\leq -1} \arrow[r] &  y_{\geq 0}[1]
    \end{tikzcd}\]
    (the right arrow is zero due to weight considerations) we see that $f$ factors through the truncation $y_{\geq 0} \rightarrow y$. Considering further
    \[\begin{tikzcd}
        & \unit \arrow[dl, dashrightarrow] \arrow[d, "f_{\geq 0}"] \arrow[rd, "0"] & \\
        y_{\geq 1} \arrow[r] & y_{\geq 0} \arrow[r] & y_0 \arrow[r] &  y_{\geq 1}[1]
    \end{tikzcd}\]
    (the right arrow is zero due to the assumption $t(f)=0$) we see that $f$ factors once more through $y_{\geq 1} \rightarrow y_{\geq 0} \rightarrow y$.

    Now looking at $g$, we obtain a map $y_{\geq 1} \rightarrow z_{\geq 2}$ factoring $g_{\geq 1}$:
    \[\begin{tikzcd}
        y_{\geq 2} \arrow[r] \arrow[d] & y_{\geq 1} \arrow[dl, dashrightarrow] \arrow[d, "g_{\geq 1}"] \arrow[r] & y_1 \arrow[r] \arrow[d, "0"] & y_{\geq 2}[1] \arrow[d] \\
        z_{\geq 2} \arrow[r] & z_{\geq 1} \arrow[r] & z_1 \arrow[r] &  z_{\geq 2}[1]
    \end{tikzcd}\]
    (that is, the diagram
    \[\begin{tikzcd}
        & y_{\geq 1} \arrow[dl, dashrightarrow] \arrow[d, "g_{\geq 1}"] \\
        z_{\geq 2} \arrow[r] & z_{\geq 1}
    \end{tikzcd}\]
    commutes).

    In all, we get a diagram as follows:
    \[\begin{tikzcd}
        \unit \arrow[r] \arrow[rd, swap, "f"] & y_{\geq 1} \arrow[r] \arrow[d] \arrow[rd, "g_{\geq 1}"] & z_{\geq 2} \arrow[d] \\
        & y \arrow[rd, swap, "g"] & z_{\geq 1} \arrow[d] \\
        & & z
    \end{tikzcd}\]
    which moreover commutes. In particular, we get that the composite $g \circ f$ factors through $\unit \rightarrow z_{\geq 2}$. But the group $\Hom(\unit,z_{\geq 2})$ is an iterated extension of the groups $\Hom(\unit,z_n)$ for the graded pieces $z_n$ of $z$ for $n \geq 2$. Since these groups are zero in $\DAM_\gm(k/k;\F_2)$, we conclude.
\end{proof}
As a consequence, we obtain the following
\begin{prop}\label{prop:tisconservativeonArtins}
    The kernel of the weight complex functor $t:\DAM(k/k;\F_2) \rightarrow \K(\mathcal{A}_\oplus)$ is trivial. In particular, the zero ideal is prime in $\DAM(k/k;\F_2)$.
\end{prop}
\begin{remark}
    We remind the reader that one needs to take caution when discussing primes in `big' tt-categories; when we say $(0)$ is prime in such a category, we simply mean that the tensor product of nonzero objects remains nonzero.
\end{remark}
\begin{proof}
    Take $X \in \ker(t)$. By \Cref{lemma:expressionofkernel}, $X$ can be expressed as the filtered (homotopy) colimit (say by a diagram of shape $I$) of compact objects $x_i$ for which there is some morphism of the diagram out of $x_i$ which is $\otimes$-nilpotent. For each $i \in I$, choose such a morphism $f_i:x_i \rightarrow x_{i'}$. Take some compact object $x \in \DAM_\gm(k/k;\F_2)$. Then for every morphism $f:x \rightarrow x_i$, the composite
    \[\begin{tikzcd}
        x \arrow[r, "f"] & x_i \arrow[r, "f_i"] & x_{i'} \arrow[r, "f_{i'}"] & x_{i''}
    \end{tikzcd}\]
    vanishes by \Cref{lemma:DAMcomposition} (the composition $f_{i'} \circ f_i$ is already zero). With the help of \Cref{lemma:homotopycolimits}, this proves the morphism group
    \[\Hom(x,X) \simeq \colim_{i \in I}\Hom(x,x_i)\]
    is zero. As $x \in \DAM_\gm(k/k;\F_2)$ was arbitrary, we deduce $X$ is zero via compact generation.
\end{proof}
\begin{remark}
    This provides another way of proving the Balmer spectrum of $\DAM_\gm(k/k;\F_2)$ is a singleton. Indeed, for $\mathcal{K},\mathcal{T}$ as above, the homological primes $\mathcal{B} \in \Spc^h(\mathcal{K})$ can be characterised as kernels $\ker(\widehat{E_\mathcal{B}} \otimes -)$ for objects $E_\mathcal{B} \in \mathcal{T}$, which satisfy $E_\mathcal{B} \otimes E_{\mathcal{B}'}=0$ for $\mathcal{B} \neq \mathcal{B}'$ (\cite[§5]{Balmer:ResidueFields}). It follows that the zero ideal being prime in the big category $\mathcal{T}$ forces the homological spectrum $\Spc^h(\mathcal{K})$ to be singleton (because $E_\mathcal{B}=0$ is forbidden). Then, the continuous surjection $\phi:\Spc^h(\mathcal{K}) \rightarrow \Spc(\mathcal{K})$ forces the Balmer spectrum of $\mathcal{K}$ to be a singleton \textit{a fortiori}.
\end{remark}
\begin{remark}\label{remark:squarestozero}
    We can show the zero ideal is \emph{not} prime in neither $\DTM(k/k;\F_2)$ nor $\DATM(k/k;\F_2)$. Indeed, the nonzero object $\Omega$ of \Cref{prop:tisnotconservative} satisfies $\Omega \otimes \Omega \simeq 0$. We can see this by considering
    \[\Omega \otimes \Omega \simeq \hocolim(\Omega \overset{r_0}{\longrightarrow} \Omega[-1] \overset{r_1[-1]}{\longrightarrow} \cdots)\]
    from which we find, for any compact $x$, the morphism group $\Hom(x,\Omega \otimes \Omega)$ is the colimit of the $(\N \times \N)$-indexed diagram
    \[\begin{tikzcd}
        \vdots & \vdots & \vdots \\
        \Hom(x,T(-1)[-4]) \arrow[r] \arrow[u] & \Hom(x,T(-1)[-5]) \arrow[r] \arrow[u] & \Hom(x,T(-2)[-8]) \arrow[r] \arrow[u] & \cdots \\
        \Hom(x,T[-1]) \arrow[r] \arrow[u] & \Hom(x,T[-2]) \arrow[r] \arrow[u] & \Hom(x,T(-1)[-5]) \arrow[r] \arrow[u] & \cdots \\
        \Hom(x,T) \arrow[r] \arrow[u] & \Hom(x,T[-1]) \arrow[r] \arrow[u] & \Hom(x,T(-1)[-4]) \arrow[r] \arrow[u] & \cdots \\
    \end{tikzcd}\]
    which clearly vanishes (each $r_i$ is such that $r_i^2=0$); compact generation tells us $\Omega \otimes \Omega \simeq 0$.
\end{remark}

\section{Generation and the Rouquier dimension}
By \Cref{cor:spectrumofDTM}, we have that the Krull dimension of $\DTM_\gm(k/k;\F_2)$ is $0$ (the Krull dimension of a tt-category being defined as the Krull dimension of its Balmer spectrum). There is another notion of dimension for triangulated categories---the \emph{Rouquier dimension}: the Rouquier dimension of a (not necessarily tensor) triangulated category $\mathcal{K}$ is the minimum $n$ for which there exists an object $a \in \mathcal{K}$ such that $\ideal{a}^\triangle_{n+1}=\mathcal{K}$, where $\ideal{a}^\triangle_m$ for a positive integer $m$ is defined (recursively) to be the smallest full subcategory of $\mathcal{K}$ generated by $a$ closed under finite sums, summands and $m-1$ many extensions, and is defined to be $\infty$ if no such $a,n$ exist. If such an $a$ exists, it is called a \emph{strong generator} for the category. A category $\mathcal{K}$ is called \emph{classically finitely generated} if there is a finite set $\mathcal{S}$ for which $\mathcal{K}=\ideal{\mathcal{S}}^\triangle$ where $\ideal{\mathcal{S}}^\triangle = \cup_{m \geq 1}\ideal{\mathcal{S}}^\triangle_m$ is the thick triangulated subcategory generated by $\mathcal{S}$. Note that this is equivalent to a single object classically generating $\mathcal{K}$ (the object $a \coloneqq \oplus_{b \in \mathcal{S}} b$ does the job). See, for instance, \cite{Orlov:RegularCategories} for precise definitions of these subcategories.

We find below that the categories $\DTM_\gm(k/k;\F_2)$, $\DATM_\gm(k/k;\F_2)$ (recall the convention of \Cref{hyp:flexiblechar0}) are not classically finitely generated. Furthermore, we note the methods apply to give the same conclusion for the global category of motives as well.

\begin{lemma}\label{lemma:generalgeneration}
    Suppose $\mathcal{K}$ is a triangulated category equipped with a bounded weight structure, and suppose there are infinitely many nonzero $x_0,x_1,x_2,\ldots \in \mathcal{K}^\heartsuit$ such that for each $a \in \mathcal{K}^\heartsuit$, we have $\Hom(a[l],x_n)=0$ for all $l \in \Z$, for all but finitely many $x_n$. Then $\mathcal{K}$ is not classically finitely generated.
\end{lemma}
\begin{proof}
    Let $a \in \mathcal{K}$. Then the bounded weight structure permits us to filter $a$ by objects $a_i[-i] \in \mathcal{K}^\heartsuit[-i]$ of weight $i$. Then for each $a_i$, there is some minimum $n_i$ for which $\Hom(a_i[l],x_n)=0$ for all $l$ and all $n>n_i$, by assumption. By choosing $N>\operatorname{max}_i\{n_i\}$ we have for each $i$, $\Hom(a_i[l-i],x_N)=0$. This implies $\Hom(a[l],x_N)=0$, as this group is an iterated extension of those we have just calculated. It follows $a$ does not classically generate $\mathcal{K}$, as $\ideal{a}^\triangle=\mathcal{K}$ implies, for any given $x \in \mathcal{K}$, $\Hom(a[l],x)=0$ for all $l \in \Z$ only if $x=0$ (see the introduction of \cite[§2]{Bondal:Generators}).
\end{proof}

\begin{prop}\label{prop:DTMnotfinitelygenerated}
    Neither of the categories $\DTM_\gm(k/k;\F_2)$ nor $\DATM_\gm(k/k;\F_2)$ are classically finitely generated.
\end{prop}
\begin{proof}
    We take as $x_n$ the isotropic Tate motive $T(n)[2n]$. Then as the isotropic motivic cohomology of the point $H_{iso}^{p,q}(\Spec(k),\F_2)$ is concentrated in the left half of the plane, we have $\Hom_{\DTM_\gm(k/k;\F_2)}(T(i)[l],T(j))=0$ for each $i<j$ and all $l \in \Z$, and so we conclude by \Cref{lemma:generalgeneration}.
\end{proof}
We also see that the methods above apply to give a proof of the following
\begin{prop}\label{prop:DMnotfinitelygenerated}
    The category $\DM_\gm(k)$ is not classically finitely generated.
\end{prop}
\begin{proof}
    $\DM_\gm(k)$ is generated by the motives of smooth $k$-schemes and so every object is a retract of a finite iterated extension of motives of the form $X=M(V)(m)[n]$ where $V$ is smooth and $m,n \in \Z$. We have the vanishing of the motivic cohomology $H_\mathcal{M}^{p,q}(V;\Z)=0$ for $q < 0$ (\cite[Cor.~4.2]{Mazza:MotivicCohomology}). This implies for every $j>-m$ we have $\Hom(X[i],T(-j))=0$ for all $i$. The proof of \Cref{lemma:generalgeneration} allows us to finish (note we are not assuming the existence of a weight structure here, as $k$ could have positive exponential characteristic).
\end{proof}
\begin{cor}
    The Rouquier dimensions of the categories $\DTM_\gm(k/k;\F_2)$, $\DATM_\gm(k/k;\F_2)$ (for $k$ flexible and of characteristic zero) and $\DM_\gm(k)$ are infinite.
\end{cor}
\begin{proof}
    By definition, a triangulated category has finite Rouquier dimension if and only if it has a strong generator. Now apply \Cref{prop:DTMnotfinitelygenerated} and \Cref{prop:DMnotfinitelygenerated}.
\end{proof}

\printbibliography[title=References]
\end{document}